\documentclass[11pt, a4paper,reqno]{amsart}

\usepackage[utf8]{inputenc}
\usepackage{amsmath}
\usepackage[dvipsnames]{xcolor}
\usepackage{amssymb,esint}
\usepackage{amsthm}
\usepackage{dsfont}
\usepackage[shortlabels]{enumitem}
\usepackage[english]{babel}
\usepackage{hyperref}
\usepackage[]{mdframed}

\calclayout
\theoremstyle{definition}
\newtheorem{theorem}{Theorem}[section]
\newtheorem{definition}[theorem]{Definition}
\newtheorem{remark}[theorem]{Remark}
\newtheorem{lemma}[theorem]{Lemma}
\newtheorem{corollary}[theorem]{Corollary}
\numberwithin{equation}{section}

\newcommand{\mref}[1]{\textbf{(M#1)}}
\newcommand{\abs}[1]{\left\lvert#1\right\rvert}
\newcommand{\norm}[1]{\left\lVert#1\right\rVert}
\newcommand{\R}{\mathbb R}
\newcommand{\N}{\mathbb N}
\newcommand{\A}{\mathcal A}
\newcommand{\dx}{\, \mathrm d}
\newcommand{\id}{\mathrm{Id}}
\newcommand{\cF}{\mathcal F}
\newcommand{\D}{\mathcal D}
\newcommand{\kvec}[1]{\mathbf{#1}}
\newcommand{\m}{\mathbf m}
\newcommand{\homd}{\mathsf{Q}}
\newcommand{\cW}{\mathcal W}
\newcommand{\E}{\mathcal E}

\DeclareMathOperator{\supp}{supp}
\DeclareMathOperator*{\esssup}{ess\,sup}
\newcommand{\indicator}{\mathds{1}}
\renewcommand{\L}{\operatorname{L}} 
\newcommand{\C}{\operatorname{C}} 
\renewcommand{\H}{\operatorname{H}} 
\DeclareRobustCommand{\Bdot}{\dot{\operatorname{B}}\protect{\vphantom{B}}}

\begin{document}
\allowdisplaybreaks
\date{13 May 2026}

\title{Nonlinear Kinetic Diffusion Equations with \(p\)-Growth}

\author{Helge Dietert}
\address[Helge Dietert]{Universit\'e Paris Cit\'e and Sorbonne
  Universit\'e, CNRS, IMJ-PRG, F-75006 Paris, France}
\email{helge.dietert@imj-prg.fr}

\author{Lukas Niebel}
\address[Lukas Niebel]{Institut f\"ur Analysis und Numerik, Universit\"at M\"unster, Orl\'eans-Ring 10, 48149 M\"unster, Germany.}
\email{lukas.niebel@uni-muenster.de }

\author{Rico Zacher}
\address[Rico Zacher]{Institut f\"ur Angewandte Analysis, Universit\"at Ulm, Helmholtzstra\ss{}e 18, 89081 Ulm, Germany.}
\email{rico.zacher@uni-ulm.de}

\begin{abstract}
  We establish the local boundedness of (sub-)solutions to nonlinear kinetic
  diffusion equations with \(p\)-growth, where the kinetic \(p\)-Laplace
  equation is a prototypical example.  A key ingredient is the derivation of
  kinetic Gagliardo--Nirenberg inequalities, where the Lebesgue norm of a
  function is estimated in terms of its transport and diffusive directions
  controlled in different Lebesgue spaces.
\end{abstract}

\maketitle

\section{Introduction}

We are interested in estimates for (sub-)solutions $f = f(t,x,v)$ to nonlinear
kinetic equations of the form
\begin{equation}
  \label{eq:nonlinear-laplace}
  (\partial_t + v \cdot \nabla_x)f
  - \nabla_{v} \cdot A(t,x,v,f,\nabla_v f) = 0,
\end{equation}
for variables \(t\) in \(\R\) and \(x,v\) in \(\R^d\).
Throughout, we assume that
\(A:\R^{1+2d}\times\R\times\R^d\to\R^d\) is a Carath\'eodory vector field, that is,
\(A(\cdot,\cdot,\cdot,\eta,\xi)\) is measurable for every \((\eta,\xi)\in\R\times\R^d\),
and \(A(t,x,v,\cdot,\cdot)\) is continuous for a.e. \((t,x,v)\).  We also suppose the
homogeneous bounds
\begin{equation} \label{eq:a-bounds}
  A(t,x,v,\eta,\xi)\cdot \xi \ge \lambda |\xi|^p,\qquad
  |A(t,x,v,\eta,\xi)| \le \Lambda |\xi|^{p-1},
\end{equation}
for a.e. \((t,x,v)\in\R^{1+2d}\), all \(\eta\in\R\), and all \(\xi\in\R^{d}\),
for some \(0<\lambda,\Lambda<\infty\) and \(p \in (1,\infty)\).

We say that \eqref{eq:nonlinear-laplace} is a nonlinear kinetic (local) diffusion equation with $p$-growth.
A concrete example is the kinetic $p$-Laplace equation
\begin{equation}
  \label{eq:kinetic-plaplace}
  (\partial_t + v \cdot \nabla_x)f
  - \nabla_{v} \cdot (  |\nabla_v f|^{p-2} \nabla_{v} f) = 0.
\end{equation}

\medskip

For the linear case $A(t,x,v,f,\nabla_v f) = \mathfrak{a}(t,x,v) \nabla_v f$, where
$\mathfrak a \in \L^\infty(\R^{1+2d} ; \R^{d \times d})$ is an elliptic diffusion matrix, a
kinetic De~Giorgi--Nash--Moser theory has been established over the last years
\cite{MR2773175,golse-imbert-mouhot-vasseur-2019-harnack-fokker-planck-landau,
  MR4453413,dietert-hirsch-2022-preprint-regularity,
  dietert2025criticaltrajectorieskineticgeometry,dn_kineticnsf_2025}.  A first
step is the local boundedness of subsolutions, first established in
\cite{MR2068847}.  There, a crucial step is the gain of integrability, where the
comparison with the fundamental solution of the constant coefficient problem has
been used.  Later variants, including Sobolev-type embeddings, can be found in
\cite{golse-imbert-mouhot-vasseur-2019-harnack-fokker-planck-landau,AIN,AN_weakLp_2025}.

In \cite{MR4468371} these results were framed in the nonlinear setting with
\(p=2\).  There, the gain of integrability in the case of nonlinear kinetic
equations with local diffusion of general $p$-growth is stated as a challenging
problem in Section~6, Problem~2.  Similarly in \cite[p. 22]{MR4950599}, while
referring to the case of non-local diffusion, it is said that ``it is not easy
to overcome the obstacles posed by both the lack of ellipticity and by the
$p$-growth''.

\medskip

In this work, we extend the gain of integrability to nonlinear kinetic diffusion
equations with $p$-growth by adapting the method of kinetic trajectories, namely
curves suited to the geometry of the vector fields $\partial_t + v \cdot \nabla_x$ and $\nabla_v$, as
proposed and studied in a series of works
\cite{MR4875497,anceschi2025poincareinequalityquantitativegiorgi,dietert2025criticaltrajectorieskineticgeometry,dn_kineticnsf_2025}.

Our main result is the following.
\begin{theorem}\label{thm:localbdd}
  Let $p \in (2-\frac{2}{3d+2},2+\frac{2}{d})$. Then any nonnegative weak
  subsolution $f \in \L^\infty_t\L^2_{x,v} \cap \L^p_{t,x}\H^{1,p}_v$ to
  \eqref{eq:nonlinear-laplace} is locally bounded.
\end{theorem}

From the formal energy estimate we obtain control of \(\nabla_v f \in \L^p_{t,x,v}\)
and using the equation~\eqref{eq:nonlinear-laplace} again, we find that
\begin{equation*}
  (\partial_t + v \cdot \nabla_x)f = \nabla_v \cdot S_0
\end{equation*}
for $S_0 \in \L^{\frac{p}{p-1}}_{t,x,v}$.

The key step is then to establish a corresponding kinetic Gagliardo--Nirenberg
estimate with distinct Lebesgue exponents. We believe this result is of independent interest.

\begin{theorem}\label{thm:embedding}
  Let \(p,\mu\in(1,\infty)\) and define
  \[
    \frac1q
    :=
    \frac{1}{4d+2}
    \left(
    \frac{3d+1}{p}
    +
    \frac{d+1}{\mu}
    -1
    \right).
  \]
  Assume that
  \[
    q>\max\{2,p,\mu\}.
  \]
  Then there exists a constant \(C=C(d,p,\mu)>0\) such that for any
  \(f\in \L^2(\R^{1+2d})\) with
  \(\nabla_v f\in \L^p(\R^{1+2d})\) and satisfying
  \[
    (\partial_t+v\cdot\nabla_x)f=\nabla_v\cdot S_0
  \]
  in the distributional sense for some
  \(S_0\in \L^\mu(\R^{1+2d};\R^d)\), it holds that
  \begin{equation} \label{kinGagNir}
    \|f\|_{\L^q_{t,x,v}}
    \le
    C
    \|\nabla_v f\|_{\L^p_{t,x,v}}^{\frac{3d+1}{4d+2}}
    \|S_0\|_{\L^\mu_{t,x,v}}^{\frac{d+1}{4d+2}} .
  \end{equation}
\end{theorem}

For \(\mu=p'=\frac p{p-1}\), the relation above reduces to
\[
  q=p\,\frac{4d+2}{d(p+2)},
\]
and the admissibility condition $q>\max\{2,p,p'\}\,(=\max\{p,p'\})$ is equivalent to
\[
  2-\frac2{3d+2}<p<2+\frac2d.
\]

\begin{remark}[Scaling considerations]\label{rem:scaling-embedding}
  The interpolation inequality \eqref{kinGagNir} is sharp by scaling.  Indeed, suppose we have
  an interpolation estimate of the form
  \begin{equation} \label{geninter}
    \norm{f}_q \lesssim \norm{\nabla_v f}_p^{\alpha}\; \norm{S_0}_\mu^{1-\alpha}
  \end{equation}
  for some \(\alpha \in [0,1]\) and integrabilities \(q,p,\mu\).  We can now scale the
  function \(f\) in two ways:
  \begin{equation*}
    f_{\lambda}(t,x,v) = f(t,\lambda x, \lambda v)\qquad\text{and}\qquad
    f_{\nu}(t,x,v) = f(\nu t,\nu x, v).
  \end{equation*}
  Applying the inequality \eqref{geninter} after the rescaling yields the conditions
  \begin{equation*}
    \lambda^{-\frac{2d}{q}} = \lambda^{\left( 1-\frac{2d}{p} \right)\alpha}
    \lambda^{\left(-1-\frac{2d}{\mu}  \right)(1-\alpha)},\quad \lambda>0,
  \end{equation*}
  and
  \begin{equation*}
    \nu^{-\frac{(d+1)}{q}}
    = \nu^{- \frac{(d+1)}{p} \alpha}
    \nu^{\left( 1 - \frac{d+1}{\mu} \right)(1-\alpha)},\quad \nu>0.
  \end{equation*}
  Solving these two identities for $\alpha$ and $q$ gives
  \[
    \alpha=\frac{3d+1}{4d+2},
    \qquad
    1-\alpha=\frac{d+1}{4d+2},
  \]
  and the integrability relation
  \[
    \frac1q
    =
    \frac{1}{4d+2}
    \left(
    \frac{3d+1}{p}
    +
    \frac{d+1}{\mu}
    -1
    \right).
  \]
\end{remark}

\begin{remark}[Restriction on \(p\) in Theorem~\ref{thm:embedding}]
  In the proof, we express \(f\) as a sum of kinetic convolutions involving
  \(\nabla_v f\) and \(S_0\). The Young inequalities for these two terms require
  \[
    q>p,
    \qquad
    q>\mu.
  \]
  These two conditions are equivalent to
  \[
    -\frac1{d+1}<\frac1p-\frac1\mu<\frac1{3d+1}.
  \]
  Moreover, since the global estimate starts from \(f\in \L^2\) and the mollified
  term is sent to zero as \(\tau\to\infty\), we impose \(q>2\). Thus the concise
  admissibility condition is
  \[
    q>\max\{2,p,\mu\}.
  \]
  In the localised setting, we expect that this bound can be relaxed, since on a finite-measure set one can control \(\norm{S_0}_{L^{\tilde \mu}}\) by \(\norm{S_0}_{L^\mu}\) for any \(\tilde \mu\le \mu\), and then run the estimate with \(\tilde \mu\) in place of \(\mu\).
\end{remark}

\begin{remark}
  Theorem~\ref{thm:embedding} is stated with independent integrability exponents
  for \(\nabla_v f\) and \(S_0\). The case relevant for the kinetic \(p\)-Laplace
  equation corresponds to \(\mu=p'\), since then
  \(S_0=A(t,x,v,f,\nabla_v f)\) has the growth
  \(|S_0|\lesssim |\nabla_v f|^{p-1}\).
\end{remark}

\begin{remark}
  The proof can be adapted to cover mixed Lebesgue spaces
  \(\L^{q_1}_{t}\L^{q_2}_{x,v}\).
\end{remark}

A related question concerns the transfer of regularity under such control on the transport term and the velocity gradient. We quantify this effect using the following definition
(similar to \cite{hoermander-1967-hypoelliptic}):
\begin{definition}\label{def:besov-x}
  For $s\in (0,1)$ and $q\in [1,\infty]$, the homogeneous
  (${\Bdot^{s}_{q\infty,x}\L^q_{t,v}}$\ --) Besov norm
  in the $x$-variable is defined as
  \[
    \norm{f}_{\Bdot^{s}_{q\infty,x}\L^q_{t,v}}
    :=
    \sup_{h\neq 0}\frac{\|\Delta_x^h f\|_{\L^q_{t,x,v}}}{|h|^s},
  \]
  where the finite difference operator in the $x$-variable is given by
  \begin{equation}\label{eq:definition-diff-x}
    [\Delta_x^h f](t,x,v) := f(t,x+h,v)-f(t,x,v).
  \end{equation}
\end{definition}

Through the hypoelliptic interaction
of the transport term and the diffusion, we gain regularity in \(x\) as reflected
by the following result:
\begin{theorem}\label{thm:transfer-of-regularity}
  Let
  \[
    p \in \Bigl(2-\frac{2}{3d+2},\,2+\frac{2}{d}\Bigr),
    \qquad
    p'=\frac{p}{p-1},
    \qquad
    \bar q:=p\,\frac{4d+2}{d(p+2)}.
  \]
  Then, for every
  \[
    q\in \Bigl(\max\{p,p'\},\,\bar q\Bigr)
  \]
  there exists $C=C(d,p,q)>0$ such that the following holds.
  If $f\in \L^2(\R^{1+2d})$ with $\nabla_v f\in \L^p(\R^{1+2d})$ satisfies
  \[
    (\partial_t+v\cdot \nabla_x)f=\nabla_v\cdot S_0
  \]
  in the sense of distributions for some $S_0 \in \L^{p'}(\R^{1+2d};\R^d)$, then
  \[
    \norm{f}_{\Bdot^{s(q)}_{q\infty,x}\L^q_{t,v}}
    \le
    C
    \|\nabla_v f\|_{\L^p_{t,x,v}}^{\alpha(q)}
    \|S_0\|_{\L^{p'}_{t,x,v}}^{1-\alpha(q)}
  \]
  where
  \[
    s(q)=\frac{p(4d+2)-qd(p+2)}{q\,[d(p-2)+2p+2]}
  \]
  and
  \[
    \alpha(q) =
    \frac{3d+1}{4d+2} - \frac{(d-1)}{4d+2} s(q)
    = \frac{q[p(1+d)-d+1]-(d-1)p}{q\,[d(p-2)+2p+2]}.
  \]
\end{theorem}

\begin{remark}[Scaling considerations]
  Considering the scalings as in Remark~\ref{rem:scaling-embedding}, we find
  that for an inequality of the form
  \begin{equation*}
    \norm{f}_{\Bdot^{s(q)}_{q\infty,x}\L^q_{t,v}}
    \le
    \|\nabla_v f\|_{\L^p_{t,x,v}}^{\alpha}
    \|S_0\|_{\L^{r}_{t,x,v}}^{1-\alpha}
  \end{equation*}
  we get the scaling equations
  \begin{equation*}
    \lambda^{-\frac{2d}{q}+s} = \lambda^{\left( 1-\frac{2d}{p} \right)\alpha}
    \lambda^{\left(-1-\frac{2d}{r}  \right)(1-\alpha)}
  \end{equation*}
  and
  \begin{equation*}
    \nu^{-\frac{(d+1)}{q}+s}
    = \nu^{-\frac{(d+1)}{p} \alpha}
    \nu^{\left( 1 - \frac{d+1}{r} \right)(1-\alpha)}
  \end{equation*}
  for \(\lambda,\nu > 0\).  As before, this explains the obtained result.
\end{remark}

We can trade the obtained regularity \(s(q)\) against the gain of
integrability \(q\), i.e.\ \(s\) grows as we lower \(q\).  The lower bound
\(q > \max\{p,p'\}\) limits the range of \(s\) and we always have that
\(s < 1/3\).  In terms of \(p\), the upper limit for \(s\) has the maximal value
\(1/3\) at \(p=2\) and decreases for lower and larger \(p\).

To the best of our knowledge such a transfer-of-regularity result, where the
integrability exponents of the transport term and the diffusive term are not the
same, has not been observed before, even in the context of velocity averages.

By reducing \(s\), the gain of regularity also holds in terms of standard Sobolev
spaces in \(x\). Similarly to Theorem~\ref{thm:embedding}, source terms in
Lebesgue spaces with non-dual exponent can be handled analogously; the details
are straightforward and left to the reader.

The proof by kinetic trajectories is intrinsically local and can be formulated in such
way; see \cite{dietert2025criticaltrajectorieskineticgeometry}.  For a simple
presentation, we formulate and prove Theorems~\ref{thm:embedding} and
\ref{thm:transfer-of-regularity} over the whole space.

The paper is organised as follows.  In Section~\ref{sec:kintraj}, we will
introduce the kinetic trajectories and the related kinetic mollification
operator. This will be used in Section~\ref{sec:gain} to establish
Theorem~\ref{thm:embedding}.  The transfer-of-regularity result,
Theorem~\ref{thm:transfer-of-regularity}, is proved in the following
Section~\ref{sec:torHB}. In Section~\ref{sec:energy}, we introduce kinetic
cylinders adapted simultaneously to the kinetic geometry and the nonlinearity of
the $p$-Laplacian.  Finally, in Section~\ref{sec:iteration}, we iterate the
energy estimate together with the gain of integrability on kinetic cylinders adapted
to the intrinsic scale to prove local boundedness of subsolutions,
Theorem~\ref{thm:localbdd}.

We leave the study of the a priori H\"older estimate, as explained in
\cite{dibenedetto-1993-degenerate,urbano-2008} in the parabolic setting, as an
interesting open problem for future research.  We suggest that Poincar\'e
inequalities may also be proven via kinetic trajectories
\cite{MR4453413,MR4875497,anceschi2025poincareinequalityquantitativegiorgi}.

\subsection*{Notation}
In what follows $d\in \N$ denotes the dimension. The subscripts $t$, $x$, and $v$
refer to the first, the following $d$ and the last $d$ variables. These
subscripts are used for functions, for gradients $\nabla$, for Lebesgue $\L^p$, and
Sobolev $\H^s$ spaces in a way that should be clear from the context. We use
the standard notation for Lebesgue spaces $\L^p$ and the weak Lebesgue spaces
$\L^{p,\infty}$. The symbol $\fint$ denotes the averaged integral.  Whenever we write
$\lesssim,\gtrsim,\approx$, we estimate or compare by a universal constant. Moreover, in proofs,
$C$ denotes a universal constant, which may change from line to line.

\subsection*{Acknowledgements}
Lukas Niebel is funded by the Deutsche Forschungsgemeinschaft \linebreak (DFG, German Research Foundation) under Germany's Excellence Strategy EXC 2044/2 -- 390685587, Mathematics M\"unster: Dynamics--Geometry--Structure.

\section{Kinetic trajectories} \label{sec:kintraj}

Our approach is based on mollification in the direction of critical kinetic
trajectories as proposed in
\cite{dietert2025criticaltrajectorieskineticgeometry}.  The main novelty is
that we allow a different scaling to adapt to the nonlinearity.

All $2 \times 2$ matrices act on $\R^{2d}$ as their tensor with $\id_d$.

\subsection{Critical kinetic trajectories}

For a scaling parameter \(\beta \in (1,\infty)\) introduce the paths
\begin{align*}
  g_1(r) =  r^{\beta} \sin \log r
  \quad\text{ and }\quad
  g_2(r) =  r^{\beta} \cos \log r.
\end{align*}

Given $(t,x,v) \in \R^{1+2d}$, $m_0 \in \R \neq 0$ and $( m_1, m_2) \in \R^{2d}$ we define
the kinetic trajectory
\begin{equation*}
  \gamma^{\m} = (\gamma^{\m}_t,\gamma^{\m}_x,\gamma^{\m}_v) \colon [0,\infty) \times \R^{1+2d}\times \R^{1+2d} \to \R^{1+2d}
\end{equation*}
with $\m = (m_0,m_1,m_2)$ as
\begin{equation*}
  \gamma^{\mathbf m}(r;(t,x,v))
  :=
  \begin{pmatrix}
    t + m_0 r                             \\
    x + m_0 r v + m_1 g_1(r) + m_2 g_2(r) \\
    v + m_0^{-1} \left[ m_1 \dot g_1(r) + m_2 \dot g_2(r) \right]
  \end{pmatrix}.
\end{equation*}
Using matrices, it can be expressed as
\begin{equation*}
  \gamma^{\mathbf m}(r;(t,x,v))
  := \begin{pmatrix}
    t+m_0 r \\
    \E_{m_0}(r)\begin{pmatrix}
                 x \\
                 v
               \end{pmatrix} + \A_{m_0}(r)
    \begin{pmatrix}
      m_1 \\  m_2
    \end{pmatrix}
  \end{pmatrix}
\end{equation*}
where we have denoted, for $\delta \in \R$ and $r \in [0,\infty)$,
\begin{equation*}
  \A_{m_0}(r) = \D_{m_0}^{-1} \cW(r),\quad
  \cW(r) =
  \begin{pmatrix}
    g_1(r)      & g_2(r)      \\
    \dot g_1(r) & \dot g_2(r)
  \end{pmatrix},
  \quad \D_{\delta} :=
  \begin{pmatrix}
    \id_d & 0             \\
    0     & \delta\,\id_d
  \end{pmatrix},\quad
\end{equation*}
and
\begin{equation*}
  \E_\delta(r) =
  \begin{pmatrix}
    \id_d & \delta r \, \id_d \\
    0     & \id_d
  \end{pmatrix}.
\end{equation*}

\medskip

The mapping $\gamma^{\mathbf m}$ satisfies the following properties.
\begin{enumerate}[itemsep=0.2cm]
  \item[\hypertarget{link:M1}{\textbf{(M1)}}] It is a kinetic trajectory:
        $\dot \gamma_x ^{\m} = \dot \gamma_t ^{\m} \gamma_v ^{\m} $.
  \item[\hypertarget{link:M2}{\textbf{(M2)}}] We have
        $ \det(\cW(r)) = c_0 r^{(2\beta-1)d}$ for all $r \in [0,\infty)$ and some constant
        $c_0 \not = 0$.
  \item[\hypertarget{link:M3}{\textbf{(M3)}}] We have
        $\abs{((\A_{m_0}(r))^{-1})_{i;1}} \lesssim r^{-\beta}$
        and
        $\abs{((\A_{m_0}(r))^{-1})_{i;2}} \lesssim \abs{m_0} r^{1-\beta}$
        for $i=1,2$ and $r \in (0,\infty)$.
  \item[\hypertarget{link:M4}{\textbf{(M4)}}] The following bounds hold:
        \begin{equation*}
          \begin{cases}
            \abs{\dot{\gamma}_v^{\m}(r)}
            \lesssim \abs{m_0}^{-1} \left( \abs{m_1} + \abs{m_2} \right) r^{\beta-2}, \\
            \abs{\gamma_v^{\m}(r)-v}
            \lesssim \abs{m_0}^{-1} \left( \abs{m_1} + \abs{m_2} \right) r^{\beta-1}, \\
            \abs{\gamma_x^{\m}(r)-x-m_0vr} \lesssim
            \left( \abs{m_1} + \abs{m_2} \right) r^{\beta}
          \end{cases}
        \end{equation*}
        for all $r \in (0,\infty)$.
\end{enumerate}

The first property \mref1 ensures that the trajectory changes only along the
transport field \(\partial_t + v\cdot \nabla_x\) and the velocity directions \(\nabla_v\).  The second and third properties
show that the trajectories effectively span the space as \(\m\) varies.
They follow from direct calculations as in
\cite{dietert2025criticaltrajectorieskineticgeometry}.  The last property \mref4
controls the size and follows directly from the definition.

The resulting forcing of these trajectories is
\begin{equation*}
  \cF_{m_0}(r) = \begin{pmatrix}
    m_0^{-1} \ddot g_1(r)
     &
    m_0^{-1} \ddot g_2(r)
  \end{pmatrix} \in \R^{d \times 2d}.
\end{equation*}

\begin{remark}
  The choice $\beta = 3/2$ leads to the trajectories of our linear study in
  \cite[Section~2.1]{dietert2025criticaltrajectorieskineticgeometry}.
\end{remark}

\subsection{Phase-space kinetic mollification}

The kinetic $p$-Laplace equation \eqref{eq:kinetic-plaplace} satisfies Galilean invariance.  This is captured by the
kinetic translation group, defined for vectors $(t,x,v),(s,y,w) \in \R^{1+2d}$ via
\begin{equation}\label{eq:group-law}
  (t,x,v) \circ (s,y,w) = (t+s,x+y+sv,v+w),
\end{equation}
and inverses are given by $(t,x,v)^{-1} = (-t,-x+tv,-v)$.  The equation is then
left-invariant under the group action.

For a kernel $J \colon \R^{1+2d} \to \R^N$, we associate the integral operator
\(T_J\) along the kinetic translation group by
\begin{equation} \label{eq:defTJ}
  \begin{aligned}
    \bigl[T_J f\bigr](t,x,v)
     & = \int_{\R^{1+2d}} f((t,x,v) \circ \kvec u) \cdot J\left( \kvec u \right) \dx \kvec u      \\
     & = \int_{\R^{1+2d}} f(\kvec u) \cdot J\left( (t,x,v)^{-1} \circ \kvec u \right) \dx \kvec u
  \end{aligned}
\end{equation}
for $(t,x,v) \in \R^{1+2d}$ and $f \colon \R^{1+2d} \to \R^N$, where the integration
variable is \(\kvec u = (u_0,u_1,u_2) \in \R^{1+2d}\).  Explicitly, we find
$(t,x,v)^{-1} \circ \kvec u = (u_0-t,u_1-x-(u_0-t)v,u_2-v)$.

\medskip

For a smooth nonnegative mollification function $\psi \in \C_c^\infty(\R^{1+2d})$ with
support in
\begin{equation}\label{eq:supp-psi}
  \supp \psi \subset (-2,-1) \times B_1(0) \times B_1(0)
\end{equation}
and unit mass, we define the kinetic mollification along the trajectories for a
function \(f = f(t,x,v)\) as
\begin{equation*}
  (t,x,v) \mapsto \int_{\R^{1+2d}} f(\gamma^{\m}(\tau;(t,x,v)))\, \psi(\m) \dx \m,
\end{equation*}
where \(\tau \in (0,\infty)\) is an artificial time parameter.  The idea is to push the
coordinates along the trajectories indexed by \(\m\) and to average over those
possible parameters.

Using the kinetic group, we find that
\begin{equation*}
  \gamma^{\m}(\tau;(t,x,v))
  = (t,x,v) \circ \kvec u,\quad
  \text{where }
  \kvec u
  =
  \begin{pmatrix}
    m_0 \tau \\
    \A_{m_0}(\tau)
    \begin{pmatrix}
      m_1 \\ m_2
    \end{pmatrix}
  \end{pmatrix}.
\end{equation*}
A change of variables \(\m \to \kvec u\) brings the kinetic mollification into the
form \eqref{eq:defTJ} as
\begin{equation*}
  T_{K_\tau}(f)
  = \int_{\R^{1+2d}} f(\kvec u)\, K_\tau\left( (t,x,v)^{-1} \circ \kvec u \right) \dx \kvec u
  = \int_{\R^{1+2d}} f(\gamma^{\m}(\tau;(t,x,v)))\, \psi(\m) \dx \m
\end{equation*}
with
\begin{equation}\label{eq:kernel-conv}
  K_\tau(u_0,u_1,u_2)
  := c_0^{-1}
  \tau^{-\homd}\left(\frac{u_0}{\tau}\right)^d
  \psi\left( \frac{u_0}{\tau}, \A_{\frac{u_0}{\tau}}^{-1}(\tau)
  \begin{pmatrix}
      u_1 \\ u_2
    \end{pmatrix}
  \right)
\end{equation}
for the universal constant $c_0$ from property (M2) and \(\homd := (2\beta-1)d+1\).
The resulting operator \(T_{K_{\tau}}\) is the kinetic mollification.

\medskip

We obtain estimates on \(f\) by controlling the difference $f-T_{K_\tau}f$.  Let us
first consider the case of smooth and compactly supported functions
$f \in \C_c^\infty(\R^{1+2d})$, whose drift can be expressed as
\begin{equation} \label{eq:kinS0S1}
  (\partial_t +v \cdot \nabla_x)f = \nabla_v \cdot S_0 + S_1
\end{equation}
where $S_0 \in \C_c^\infty(\R^{1+2d};\R^d)$ and $S_1 \in \C_c^\infty(\R^{1+2d})$.

By the fundamental theorem of calculus, following the kinetic trajectories
yields
\begin{align*}
  f(t,x,v) & -[T_{K_\tau}(f)](t,x,v)
  =  \int_{\R^{1+2d}} \Big(f(t,x,v) - f(\gamma^{\m}(\tau;(t,x,v))) \Big)
  \psi(\m)  \dx \m                                                                                                                    \\
           & = - \int_{\R^{1+2d}} \int_0^\tau \frac{\dx }{\dx r} f(\gamma^{\m}(r;(t,x,v))) \dx r \ \psi(\m)  \dx \m                   \\
           & = - \int_{\R^{1+2d}} \int_0^\tau \Big(m_0 [(\partial_t +v\cdot \nabla_x) f](\gamma^{\m}(r;(t,x,v)))                      \\
           & \hspace{2.5cm} + \dot{\gamma}^{\m}_v(r;(t,x,v)) \cdot [\nabla_v f](\gamma^{\m}(r;(t,x,v)))\Big)\dx r \ \psi(\m)  \dx \m.
\end{align*}

For the first term, we substitute \((\partial_t + v \cdot \nabla_x) f\) using \eqref{eq:kinS0S1} and find
for the resulting \(S_0\)-term, by the change of variables
\(\kvec{\tilde m} = \gamma^{\m}(r;(t,x,v))\) and integration by parts,
\begin{align*}
   & \int_{\R^{1+2d}} \int_0^\tau m_0[\nabla_v \cdot S_0](\gamma^{\m}(r;(t,x,v)))\,
  \psi(\m) \dx r \dx \m                                                                          \\
   & =\int_0^\tau \int_{\R^{1+2d}}
  \left(\frac{\tilde{m}_0-t}{r}\right)^{1+d}
  r^{-\homd} c_0^{-1}
  [\nabla_{\tilde{m}_2} \cdot S_0](\tilde{m}_0,\tilde{m}_1,\tilde{m}_2)                          \\
   & \hspace{3cm} \cdot \psi\left( \frac{\tilde{m}_0-t}{r} ,\A_{\frac{\tilde{m}_0-t}{r}}(r)^{-1}
  \left( \binom{\tilde{m}_1}{\tilde{m}_2}
    - \E_{\frac{\tilde{m}_0-t}{r}}(r) \binom{x}{v}  \right) \right)
  \dx \kvec{\tilde m} \dx r                                                                      \\
   & =-\int_0^\tau \int_{\R^{1+2d}}
  \left( \frac{\tilde{m}_0-t}{r}\right)^{1+d}
  r^{-\homd}c_0^{-1}
  S_0(\tilde{m}_0,\tilde{m}_1,\tilde{m}_2)                                                       \\
   & \quad \cdot  \left[ [\nabla_{x,v}\psi]^T
    \left( \frac{\tilde{m}_0-t}{r} ,
    \A_{\frac{\tilde{m}_0-t}{r}}(r)^{-1}
    \left( \binom{\tilde{m}_1}{\tilde{m}_2}-\E_{\frac{\tilde{m}_0-t}{r}}(r) \binom{x}{v}  \right) \right)(\A_{ \frac{\tilde{m}_0-t}{r} }^{-1}(r))_{\cdot;2} \right]
  \dx \kvec{\tilde m} \dx r.
\end{align*}

Performing the same change of variables also in the term involving $S_1$ and in the
second term, we find the representation
\begin{equation} \label{eq:repfTtauf}
  (f-T_{K_\tau} f)(t,x,v) = \int_0^\tau\Big( [T_{G^0_r}(S_0)](t,x,v)
  + [T_{G^1_r}(S_1)](t,x,v)
  + [T_{G^v_r} (\nabla_v f)](t,x,v)\Big) \dx r
\end{equation}
with the kinetic kernels defined for $r \in (0,\infty) $ as
\medskip
\begin{enumerate}
  \item $G^0_r \colon \R^{1+2d} \to \R^d$,
        \begin{equation*}
          G^0_r(s,y,w)
          = c_0^{-1} \left(\frac{s}{r}\right)^{d+1} r^{-\homd}
            [\nabla_{x,v}\psi]^T\left( \frac{s}{r} ,\A_{\frac{s}{r}}(r)^{-1}  \binom{y}{w} \right)
          \left( \left( \A_{ \frac{s}{r} }(r) \right)^{-1}\right)_{\cdot;2}
        \end{equation*}
  \item $G^1_r \colon  \R^{1+2d} \to \R$,
        \begin{equation*}
          G^1_r(s,y,w)
          = -c_0^{-1}  \left(\frac{s}{r}\right)^{d+1} r^{-\homd}
          \psi\left( \frac{s}{r}, \A_{\frac{s}{r}}(r)^{-1} \binom{y}{w}\right) ,
        \end{equation*}
  \item and $G^v_r \colon \R^{1+2d} \to \R^d$,
        \begin{equation*}
          G^v_r(s,y,w)
          = -c_0^{-1} \left(\frac{s}{r}\right)^{d} r^{-\homd}
          \psi\left( \frac{s}{r}, \A_{\frac{s}{r}}(r)^{-1} \binom{y}{w} \right)
          \cF_{\frac{s}{r}}(r)\A_{\frac{s}{r}}(r)^{-1}   \binom{y}{w},
        \end{equation*}
\end{enumerate}
\medskip
for $(s,y,w) \in \R^{1+2d}$.

\begin{remark} \label{rem:directionoftime}
  We used the solution property, i.e.\ $f$ solves \eqref{eq:kinS0S1}, to control the $(\partial_t +v \cdot \nabla_x)$ direction of $f$ in terms of $S_0$ and $S_1$. If $f$ is only a sub- or supersolution, then the equality becomes an inequality, and one needs to adapt the direction of time of the kinetic trajectory. Here, we choose $\supp \psi \subset (-2,-1) \times B_1(0) \times B_1(0)$, and thus the trajectories are in the negative time direction and adapted to subsolutions.
\end{remark}

\subsection{\texorpdfstring{\(\L^p\)}{Lp} bounds for the kinetic mollification}

The kinetic integral operator \(T_J\) from \eqref{eq:defTJ} is not a symmetric
convolution.  Nevertheless, we have Young-type inequalities due to the underlying kinetic group structure.

\begin{lemma}\label{lem:young}
  Let $1\le p,q,\theta \le \infty$ with $\frac{1}{q}+1 = \frac{1}{\theta}+\frac{1}{p}$.  For any
  $J \in \L^{\theta}(\R^{1+2d};\R^N)$ we have
  \begin{equation*}
    \norm{T_J(f)}_{\L^q} \le \norm{J}_{\L^{\theta}} \norm{f}_{\L^p}
  \end{equation*}
  for all $f \in \L^p(\R^{1+2d};\R^N)$.
\end{lemma}

\begin{proof}
  Use \cite[Theorem 1.2.12]{MR2445437}.
\end{proof}

Moreover, from the classical weak Young inequality we deduce the following
central estimate.

\begin{lemma}\label{lem:weakYoung}
  Let $1 < p,q,\theta < \infty$ with $\frac{1}{q}+1 = \frac{1}{\theta}+\frac{1}{p}$.  There
  exists \(C>0\) such that for any $J \in \L^{\theta,\infty}(\R^{1+2d};\R^N)$, we have
  \begin{equation*}
    \norm{T_J(f)}_{\L^q} \le C \norm{J}_{\L^{\theta,\infty}} \norm{f}_{\L^p}
  \end{equation*}
  for all $f \in \L^p(\R^{1+2d};\R^N)$.
\end{lemma}

\begin{proof}
  Use \cite[Theorem 1.4.25]{MR2445437}.
\end{proof}

\medskip

From the definition of the kernels, we can directly deduce a bound on the
support and the supremum of the kernels.
\begin{lemma}\label{lem:kernel-support-sup}
  There exists $C>0$ such that for every $r>0$,
  \begin{multline*}
    \supp K_r \cup \supp G^0_r \cup \supp G^1_r \cup \supp G^v_r\\
    \subset
    \Bigl\{(s,y,w)\in \R^{1+2d} : -2r \le s \le -r,\ |y| \le C r^\beta,\ |w| \le Cr^{\beta-1}\Bigr\}.
  \end{multline*}
  Moreover,
  \begin{align*}
    |K_r|            & \lesssim r^{-\homd},
                     &
    |\nabla_y K_r|   & \lesssim r^{-\beta-\homd},
    \\
    |G^0_r|          & \lesssim r^{1-\beta-\homd},
                     &
    |\nabla_y G^0_r| & \lesssim r^{1-2\beta-\homd},
    \\
    |G^1_r|          & \lesssim r^{-\homd},
                     &
    |\nabla_y G^1_r| & \lesssim r^{-\beta-\homd},
    \\
    |G^v_r|          & \lesssim r^{\beta-2-\homd},
                     &
    |\nabla_y G^v_r| & \lesssim r^{-2-\homd}.
  \end{align*}
  Here $\nabla_y$ denotes differentiation in the spatial kernel variable $y$.
\end{lemma}
\begin{proof}
  The result follows directly from the definitions \eqref{eq:kernel-conv},
  \eqref{eq:repfTtauf} and the support of \(\psi\) in \eqref{eq:supp-psi}.
\end{proof}

We can directly conclude the following bound.
\begin{corollary}\label{lem:estK}
  Let $\theta \in [1,\infty]$. We have
  \begin{equation*}
    \norm{K_r}_{\L^\theta} \lesssim  r^{\homd\left( \frac{1}{\theta}-1 \right)},\quad r>0.
  \end{equation*}
\end{corollary}
\begin{proof}
  Note that, by Lemma \eqref{lem:kernel-support-sup}, \(|\supp K_r| \lesssim r^{\homd}\). Combining this with the upper bound for $K_r$ yields the result.
\end{proof}

We next derive bounds for the integrated kernels from \eqref{eq:repfTtauf}.

\begin{lemma}\label{lem:estintK}
  Let \(\beta \in (1,2)\) and
  \begin{equation*}
    \theta_0 = \frac{\homd}{\homd + \beta - 2},\quad
    \theta_1 = \frac{\homd}{\homd - 1},\quad
    \theta_v = \frac{\homd}{\homd + 1 - \beta}.
  \end{equation*}
  Then, uniformly in $\tau>0$,
  \begin{equation*}
    \norm{\int_0^{\tau} G^0_r \dx r}_{\L^{\theta_0,\infty}}
    \lesssim 1,\quad
    \norm{\int_0^{\tau} G^1_r \dx r}_{\L^{\theta_1,\infty}}
    \lesssim 1,\quad
    \norm{\int_0^{\tau} G^v_r \dx r}_{\L^{\theta_v,\infty}}
    \lesssim 1.
  \end{equation*}
\end{lemma}

\begin{proof}
  If we have \((s,y,w) \in \supp G^0_r\) with $r>0$, we find directly from the support that $s<0$,
  \(r \in (-s/2,-s)\) and \(|y| \lesssim (-s)^{\beta}\) as well as \(|w| \lesssim (-s)^{\beta-1}\).  Using the upper
  bound we therefore find that, uniformly in $\tau>0$,
  \begin{align*}
    \abs{ \left( \int_0^{\tau} G^0_r \dx r \right) (s,y,w) }
     & \lesssim
    \int_{\max\{0,-s/2\}}^{\min\{\tau,-s\}}
    r^{1-\beta-\homd} \dx r \,
    \indicator_{s < 0} \indicator_{|y|\lesssim (-s)^{\beta}, |w|\lesssim (-s)^{\beta-1}} \\[0.4em]
     & \lesssim
    (-s)^{2-\beta-\homd}
    \indicator_{s < 0}
    \indicator_{|y|\lesssim (-s)^{\beta}, |w|\lesssim (-s)^{\beta-1}}.
  \end{align*}
  The claimed bound for \(G^0_r\) then follows directly from the definition of
  the weak \(\L^{\theta,\infty}\) spaces.

  For the other kernels, this follows in exactly the same way, using the corresponding upper bounds
  in Lemma~\ref{lem:kernel-support-sup}.
\end{proof}

\medskip

For the transfer of regularity, we need to understand how the
difference operator \eqref{eq:definition-diff-x} acts with the integral operator.
\begin{lemma}\label{lem:diff-kinconv}
  Under suitable integrability conditions, we have that for every $J \colon \R^{1+2d}\to \R^N$, every $g \colon \R^{1+2d}\to \R^N$, and every $h\in \R^d$,
  \[
    \Delta_x^h[T_J(g)] = T_{\Delta_y^{-h}J}(g).
  \]
\end{lemma}
\begin{proof}
  Using \eqref{eq:defTJ}, we compute
  \begin{align*}
     & [T_J(g)](t,x+h,v)-[T_J(g)](t,x,v)                                     \\
     & = \int_{\R^{1+2d}} g(\kvec u)\cdot \Bigl[
    J\bigl(u_0-t,u_1-x-h-(u_0-t)v,u_2-v\bigr)                                \\
     & \hspace{3cm}-J\bigl(u_0-t,u_1-x-(u_0-t)v,u_2-v\bigr)\Bigr]\dx \kvec u \\
     & = [T_{\Delta_y^{-h}J}(g)](t,x,v).
    \qedhere
  \end{align*}
\end{proof}

By means of the gradient bounds from Lemma~\ref{lem:kernel-support-sup}, we can deduce similar
estimates for the difference.

\begin{lemma}\label{lem:kernel-tor}
  Let $\theta\in [1,\infty]$, \(s \in (0,1)\) and $h\in \R^d$.  Then
  \begin{equation*}
    \norm{\Delta_y^{-h}K_r}_{\L^\theta} \lesssim  |h|^s r^{\homd\left( \frac{1}{\theta}-1 \right)-\beta s},\quad r>0.
  \end{equation*}
  Moreover, if $s\in (0,\frac{1}{3})$, $\beta \in (\frac{1}{1-s},\frac{2}{1+s})$ and
  \begin{equation*}
    \theta_0(s) = \frac{\homd}{\homd + \beta - 2 + \beta s},\quad
    \theta_1(s) = \frac{\homd}{\homd - 1 + \beta s},\quad
    \theta_v(s) = \frac{\homd}{\homd + 1 - \beta + \beta s},
  \end{equation*}
  then, uniformly in $\tau>0$,
  \begin{align*}
    \norm{\int_0^{\tau} \Delta_y^{-h}G^0_r \dx r}_{\L^{\theta_0(s),\infty}}
     & \lesssim |h|^s, \\
    \norm{\int_0^{\tau} \Delta_y^{-h}G^1_r \dx r}_{\L^{\theta_1(s),\infty}}
     & \lesssim |h|^s, \\
    \norm{\int_0^{\tau} \Delta_y^{-h}G^v_r \dx r}_{\L^{\theta_v(s),\infty}}
     & \lesssim |h|^s.
  \end{align*}
\end{lemma}
\begin{proof}
  Note that, for $|h|\le r^\beta$, we can use the gradient estimate in Lemma~\ref{lem:kernel-support-sup} to get the bound
  \[
    |\Delta_y^{-h}K_r|\lesssim r^{-\beta- \homd}|h|  \le |h|^s r^{-\homd - \beta s}.
  \]
  For $|h|> r^\beta$, we simply use the triangle inequality and the bound for $K_r$ to obtain
  \[
    |\Delta_y^{-h}K_r|\lesssim r^{- \homd}\le  |h|^s r^{-\homd - \beta s}.
  \]
  With this bound for $\Delta_y^{-h}K_r$ at hand, the first assertion of the lemma follows by the same argument as in the proof of Corollary \ref{lem:estK}. Likewise, we obtain the bounds for
  the integrated kernels.
\end{proof}

\section{Gain of integrability}
\label{sec:gain}

We can now prove the gain of integrability.
\begin{proof}[Proof of Theorem~\ref{thm:embedding}]
  We write $f = f-T_{K_\tau}(f) + T_{K_\tau}(f)$.  We first estimate the difference
  \(f - T_{K_{\tau}} f\) by using the representation \eqref{eq:repfTtauf}, Lemma~\ref{lem:weakYoung} and
  Lemma~\ref{lem:estintK}. Note that for an $r$-dependent kernel $J_r$,
  \[
    \int_0^\tau T_{J_r}(g)\dx r=T_{\int_0^\tau J_r \dx r}(g).
  \]

  Starting with the forcing and assuming $\beta>1$, we find that
  \begin{align*}
    \norm{\int_0^\tau T_{G^v_r} (\nabla_v f) \dx r}_{\L^q} & \lesssim \norm{\nabla_v f}_{\L^p}
  \end{align*}
  with
  \begin{equation*}
    \frac{1}{q}+1 = \frac{1}{\theta_v}+\frac{1}{p}
    \quad\text{and}\quad
    \frac{1}{\theta_v} = 1 + \frac{1-\beta}{\homd},
  \end{equation*}
  so that
  \begin{equation*}
    \frac{1}{q} = \frac{1-\beta}{\homd} + \frac{1}{p}.
  \end{equation*}

  Similarly, we estimate the drift term, assuming in addition that $\beta<2$,
  \[
    \left\|
    \int_0^\tau T_{G^0_r}(S_0)\dx r
    \right\|_{\L^q}
    \lesssim
    \|S_0\|_{\L^\mu}
  \]
  with
  \[
    \frac1q+1=\frac1{\theta_0}+\frac1\mu
  \]
  and
  \[
    \frac1{\theta_0}=1+\frac{\beta-2}{\homd}.
  \]
  Hence
  \[
    \frac1q=\frac1\mu+\frac{\beta-2}{\homd}.
  \]

  We choose \(\beta\) so that the two estimates give the same value of \(q\), i.e.
  \[
    \frac1p+\frac{1-\beta}{\homd}
    =
    \frac1\mu+\frac{\beta-2}{\homd},
    \qquad
    \homd=(2\beta-1)d+1.
  \]
  Writing
  \[
    a:=\frac1p-\frac1\mu,
  \]
  this becomes
  \[
    2\beta-3=a\bigl((2\beta-1)d+1\bigr).
  \]
  Assuming $ad<1$, we therefore obtain
  \begin{equation} \label{betaChoice}
    \beta
    =
    \beta_{p,\mu}
    :=
    \frac{3+(1-d)a}{2(1-da)}.
  \end{equation}
  For this choice,
  \[
    \homd=\homd_{p,\mu}
    :=
    (2\beta_{p,\mu}-1)d+1
    =
    \frac{2d+1}{1-da}.
  \]
  Substituting back gives
  \[
    \frac1q
    =
    \frac1p+\frac{1-\beta_{p,\mu}}{\homd_{p,\mu}}
    =
    \frac1\mu+\frac{\beta_{p,\mu}-2}{\homd_{p,\mu}}
    =
    \frac{1}{4d+2}
    \left(
    \frac{3d+1}{p}
    +
    \frac{d+1}{\mu}
    -1
    \right).
  \]
  The conditions \(\beta_{p,\mu}>1\) and \(\beta_{p,\mu}<2\) are equivalent to
  \begin{equation} \label{abounds}
    -\frac1{d+1}
    <
    \frac1p-\frac1\mu
    <
    \frac1{3d+1},
  \end{equation}
  or equivalently \(q>p\) and \(q>\mu\). Observe that \eqref{abounds} indeed ensures that $ad<1$.

  In addition to the condition $q>\max\{p,\mu\}$, the \(\L^2\)-based
  mollification argument below requires \(q>2\).

  By Lemma \ref{lem:young} and Corollary \ref{lem:estK},
  \[
    \|T_{K_\tau}f\|_{\L^q}
    \lesssim
    \tau^{\homd_{p,\mu}\left(\frac1q-\frac12\right)}
    \|f\|_{\L^2}.
  \]
  Since \(q>2\), the exponent
  \[
    \homd_{p,\mu}\left(\frac1q-\frac12\right)
  \]
  is negative. Hence
  \[
    \|T_{K_\tau}f\|_{\L^q}\to0
    \qquad\text{as }\tau\to\infty.
  \]
  Therefore
  \[
    \|f\|_{\L^q}
    \lesssim
    \|\nabla_v f\|_{\L^p}
    +
    \|S_0\|_{\L^\mu}.
  \]

  We now obtain the multiplicative form by scaling. Observe that the assumptions of Theorem~\ref{thm:embedding} imply that either $S_0 = 0$ and $f = 0$ or $\|S_0\|_{\L^\mu}>0$. For \(\nu>0\), set
  \[
    f_\nu(t,x,v):=f(\nu t,\nu x,v),
    \qquad
    S_{0,\nu}(t,x,v):=\nu S_0(\nu t,\nu x,v).
  \]
  Then
  \[
    (\partial_t+v\cdot\nabla_x)f_\nu
    =
    \nabla_v\cdot S_{0,\nu}.
  \]
  Applying the additive estimate to \(f_\nu\) gives
  \[
    \|f\|_{\L^q}
    \lesssim
    \nu^{(d+1)\left(\frac1q-\frac1p\right)}
    \|\nabla_v f\|_{\L^p}
    +
    \nu^{1+(d+1)\left(\frac1q-\frac1\mu\right)}
    \|S_0\|_{\L^\mu}.
  \]
  Optimising in \(\nu>0\) yields
  \[
    \|f\|_{\L^q}
    \lesssim
    \|\nabla_v f\|_{\L^p}^{\alpha}
    \|S_0\|_{\L^\mu}^{1-\alpha},
  \]
  where
  \[
    \alpha
    =
    \frac{3d+1}{4d+2},
    \qquad
    1-\alpha
    =
    \frac{d+1}{4d+2}.\qedhere
  \]
\end{proof}

We collect two important consequences. The special case $S_0 = A(t,x,v,f,\nabla_v f)$
yields the following result.
\begin{corollary}\label{cor:gain-p}
  Let $p \in (2-\frac{2}{3d+2},2+\frac{2}{d})$. There exists a constant
  $C = C(p,d, \Lambda)>0$ such that for every $f \in \L^2(\R^{1+2d})$ with
  $\nabla_v f \in \L^p(\R^{1+2d})$ which is a weak solution
  to~\eqref{eq:nonlinear-laplace}, with the nonlinearity \(A\) satisfying the \(p\)-growth condition
  \eqref{eq:a-bounds}, it holds that
  \begin{equation*}
    \norm{f}_{\L^{q}_{t,x,v}} \le C \norm{\nabla_v f}_{\L^p_{t,x,v}}^{\frac{(p+2)d+p}{4d+2}}
  \end{equation*}
  where
  \begin{equation*}
    q = p\, \frac{4d+2}{d(p+2)}.
  \end{equation*}
\end{corollary}

The same proof works for subsolutions, but we need to assume that $f$ is
nonnegative. Moreover, we may include an additional source term.

\begin{theorem}\label{thm:gainsubglobal}
  Let $p \in (2-\frac{2}{3d+2},\,2+\frac{2}{d})$ and set
  \begin{equation*}
    q := p\,\frac{4d+2}{d(p+2)},
    \qquad
    r := p\, \frac{4d+2}{p(3d+2)-2d}.
  \end{equation*}
  Let $T>0$. There exists a constant $C = C(p,d)>0$ such that for every nonnegative weak subsolution $f \in \L^2((-\infty, T) \times \R^{2d})$ to
  \begin{equation*}
    (\partial_t + v \cdot \nabla_x) f \le \nabla_v \cdot S_0 + S_1\quad\text{in}\;\,(-\infty, T) \times \R^{2d}
  \end{equation*}
  with $\nabla_v f \in \L^p((-\infty, T) \times \R^{2d})$, $S_0 \in \L^{\frac{p}{p-1}}((-\infty, T) \times \R^{2d};\R^d)$, and $S_1 \in \L^r((-\infty, T) \times \R^{2d})$, we have
  \begin{align*}
     & \norm{f}_{\L^{q}((-\infty, T) \times \R^{2d})}                                                                                                                                                        \\
     & \le C \left( \norm{\nabla_v f}_{\L^p((-\infty, T) \times \R^{2d};\R^d)} + \norm{S_0}_{\L^{\frac{p}{p-1}}((-\infty, T) \times \R^{2d};\R^d)} + \norm{S_1}_{\L^r((-\infty, T) \times \R^{2d})} \right).
  \end{align*}
\end{theorem}

\begin{proof}
  This follows by the same strategy as in the proof of the gain of integrability,
  Theorem~\ref{thm:embedding}. For the new term involving $S_1$, we note that
  $r$ is chosen such that $ \frac1q+1=\frac1{\theta_1}+\frac1r$ so that
  Lemma~\ref{lem:estintK} gives the desired kernel estimate.  The direction of
  our trajectory is compatible with the subsolution property, and thus the
  identity \eqref{eq:repfTtauf} becomes an inequality in the matching direction;
  see also Remark~\ref{rem:directionoftime}.  Moreover, as the action in the
  $t$-variable is always in the negative direction of time, we may restrict the
  estimate to intervals $(-\infty,T)$.
\end{proof}

\section{Transfer of regularity in H\"ormander--Besov norms}
\label{sec:torHB}

\begin{proof}[Proof of Theorem~\ref{thm:transfer-of-regularity}]
  The proof is very similar to the proof of the gain of integrability,
  Theorem~\ref{thm:embedding}, where we additionally use
  Lemma~\ref{lem:diff-kinconv} to put the finite difference on the kernel and
  the estimates from Lemma~\ref{lem:kernel-tor} on the kernels.

  As in the proof of Theorem~\ref{thm:embedding}, we write
  \(f = f - T_{K_{\tau}} f + T_{K_{\tau}} f\).  Using the
  representation~\eqref{eq:repfTtauf}, we estimate the difference. Suppose that $s$ and $\beta$ satisfy the assumptions of the second part of Lemma~\ref{lem:kernel-tor}. For the
  forcing, we find
  \begin{align*}
    \norm{\Delta_x^h\int_0^\tau T_{G^v_r} (\nabla_v f) \dx r}_{\L^q}
    \lesssim |h|^s \norm{\nabla_v f}_{\L^p}
  \end{align*}
  with
  \begin{equation*}
    \frac{1}{q}+1 = \frac{1}{\theta_v(s)}+\frac{1}{p}
    \quad\text{and}\quad
    \frac{1}{\theta_v(s)} = 1 + \frac{1-\beta}{\homd} + \frac{\beta s}{\homd}
  \end{equation*}
  so that
  \begin{equation} \label{ExponentsCond}
    \frac{1}{q} = \frac{1-\beta}{\homd} + \frac{1}{p} + \frac{\beta s}{\homd}.
  \end{equation}

  Similarly, we estimate the drift term
  \begin{equation*}
    \norm{\Delta_x^h\int_0^\tau T_{G^0_r} (S_0)  \dx r}_{\L^q}
    \lesssim |h|^s \norm{S_0}_{\L^{\frac{p}{p-1}}}
  \end{equation*}
  with
  \begin{align*}
    \frac{1}{q}+1 = \frac{1}{\theta_0(s)}+\frac{p-1}{p}
    \quad\text{and}\quad
    \frac{1}{\theta_0(s)} = 1+ \frac{\beta-2}{\homd} + \frac{\beta s}{\homd},
  \end{align*}
  so that
  \begin{equation} \label{exponentcond}
    \frac{1}{q} = \frac{\beta-2}{\homd} + 1 - \frac{1}{p} + \frac{\beta s}{\homd}.
  \end{equation}

  From \eqref{ExponentsCond} and \eqref{exponentcond}, it follows that
  $\beta$ is as in \eqref{betaChoice}, where now $a=\frac{2}{p}-1$, that is, we have
  \begin{equation*}
    \beta = \frac{d(p-2) + 2p + 2}{2\bigl(d(p - 2) + p\bigr)}.
  \end{equation*}
  Inserting this into \eqref{exponentcond}, we obtain
  \begin{equation*}
    s = \frac{1}{q} \frac{p(4d+2)}{[d(p-2)+2p+2]}
    - \frac{d(p+2)}{d(p-2)+2p+2}.
  \end{equation*}

  We now need to verify the assumption on \(s\) and \(\beta\) for
  Lemma~\ref{lem:kernel-tor}.  From the formula for \(s\), we see that \(s\) is
  monotone decreasing as \(q\) increases, and the value \(\bar q\) is
  precisely the value for which \(s\) becomes negative.  Hence the condition
  \(q < \bar q\) ensures that \(s > 0\).

  For the maximal value of \(s\), we first consider the case that \(p \ge 2\) so
  that \(p = \max\{p,p'\}\).  Hence we need to choose \(q > p\) and thus we find for
  \(s\) the upper bound
  \begin{equation*}
    s(p) = \frac{2d+2-dp}{-2d+2+p(2+d)}.
  \end{equation*}
  From this formula, we see directly that \(s(p)\) is decreasing when \(p\)
  is increasing.  Moreover, it attains the maximal value \(1/3\) at \(p=2\).

  In the other case, \(p<2\), we find that \(p'=\max\{p,p'\}\) and thus we
  find the upper bound
  \begin{align*}
    s(p') & = \frac{(p-1)(4d+2)-d(p+2)}{d(p-2)+2p+2}      \\
          & = \frac{1}{3} - \left( \frac{8d+4}{3} \right)
    \left( \frac{2-p}{2-2d+p(2+d)} \right).
  \end{align*}
  In the allowed range of \(p\), we always have \(2-2d+p(2+d) > 0\) so that this
  formula shows that \(s(p') \le 1/3\).

  Hence in both cases, we find that \(s \in (0,1/3)\) as required in
  Lemma~\ref{lem:kernel-tor}.

  For the value of \(\beta\), note that it is as in the proof of
  Theorem~\ref{thm:embedding} and thus \(\beta > 1/(1-s)\).  The upper bound
  \(\beta < 2/(1+s)\) is expanded into
  \begin{equation*}
    2 > \beta(1+s)
    = \frac{\frac{p}{q}(4d+2)-4d+2p+2}{2[d(p-2)+p]}
  \end{equation*}
  and this is equivalent to
  \begin{equation*}
    p \left( 1 - \frac{1}{q} \right) > 1.
  \end{equation*}
  This is ensured by \(q > \max\{p,p'\}\) and thus Lemma~\ref{lem:kernel-tor} can
  be applied.

  The result is the claimed decay
  \[
    \| \Delta_x^h (f - T_{K_{\tau}} f)\|_{\L^{q}} \lesssim |h|^s \left(\norm{\nabla_v f}_{\L^p}+ \norm{S_0}_{\L^{\frac{p}{p-1}}} \right).
  \]

  For the mollified term, we apply Lemma~\ref{lem:diff-kinconv}, Lemma~\ref{lem:young}, and the first part of Lemma~\ref{lem:kernel-tor},
  \[
    \norm{\Delta_x^h T_{K_\tau}f}_{\L^q}
    =
    \norm{T_{\Delta_y^{-h}K_\tau}(f)}_{\L^q}
    \lesssim
    |h|^s \tau^{\homd\left(\frac1q-\frac12\right)-\beta s}\norm{f}_{\L^2}.
  \]
  Because \(q>2\) and \(\beta s>0\), the exponent is negative; hence this contribution vanishes as $\tau\to\infty$.
  Combining this with the estimate for \(f-T_{K_\tau}f\), dividing by \(|h|^s\), taking the supremum over \(h\neq0\), and letting \(\tau\to\infty\), we obtain
  \[
    \norm{f}_{\Bdot^{s}_{q\infty,x}\L^q_{t,v}}
    \lesssim
    \norm{\nabla_v f}_{\L^p}
    +
    \norm{S_0}_{\L^{\frac{p}{p-1}}}.
  \]
  The multiplicative form then follows by the same scaling argument as above.
\end{proof}

\section{Intrinsic scaling and energy estimate}
\label{sec:energy}

\subsection{\texorpdfstring{Kinetic \(p\)-cylinders}{Kinetic p-cylinders}}

Recall the group law~\eqref{eq:group-law}. The kinetic $p$-Laplace equation \eqref{eq:kinetic-plaplace} is left-invariant under the group, i.e.\ if
$f$ solves \eqref{eq:kinetic-plaplace} in an open set $\Omega \subset \R^{1+2d}$, then
\[
  g(t,x,v) := f\bigl((t_0,x_0,v_0)\circ(t,x,v)\bigr)
\]
solves \eqref{eq:kinetic-plaplace} in the translated domain $z_0^{-1}\circ\Omega$.  This
can be verified by direct computation. For the more general equation~\eqref{eq:nonlinear-laplace}, this invariance property no longer holds. Nevertheless, the transformed function $g$ solves a structurally equivalent equation
\begin{equation*}
  (\partial_t + v \cdot \nabla_x)g
  - \nabla_{v} \cdot \tilde{A}(t,x,v,g,\nabla_v g) = 0
\end{equation*}
on $z_0^{-1}\circ\Omega$. By structurally equivalent, we mean that $\tilde{A}$ satisfies the same structural assumptions \eqref{eq:a-bounds} as $A$, with the same constants $\lambda$ and $\Lambda$. By this observation, it suffices to
state and obtain local estimates for cylinders centred at the origin.

For the \(p\)-scaling of the nonlinearity~\eqref{eq:a-bounds}, the natural family of dilations, parametrized by $r>0$, is given by
\[
  \delta_r(t,x,v) := (r^p t,\; r^{1+p}x,\; r v).
\]
If $f$ solves \eqref{eq:kinetic-plaplace} in $\Omega$, then the rescaled function $f_r := f\circ \delta_r$ solves the
same equation in $\delta_r^{-1}\Omega$. In the case of the more general equation~\eqref{eq:nonlinear-laplace}, $f_r$ solves the structurally equivalent equation with nonlinearity
$\tilde{A}(t,x,v,\eta,\xi):=r^{p-1} A(r^p t,r^{1+p} x,r v,\eta, \xi/r)$.
This fixes the parabolic time scale $t\sim r^p$,
the velocity scale $v\sim r$, and the transport scale in $x$ given by $x\sim r^{1+p}$.

\medskip\noindent For suitable localisation, we impose a cutoff following the
transport flow \(\partial_t + v \cdot \nabla_x\) by using the variable \(x-tv\).  Following
\cite{dietert-hirsch-2022-preprint-regularity}, this motivates the following
definition of cylinders, where the lateral boundary is aligned so that the
transport term \(\partial_t + v \cdot \nabla_x\) commutes with the localisation in \((x,v)\).

\begin{definition}[Backward kinetic $p$-cylinder]\label{def:kinetic-cylinder}
  For \(p \in (1,\infty)\), given $\theta>0$ and $R>0$, we define the \emph{backward kinetic
    $p$-cylinder} centred at $(0,0,0)$ by
  \begin{equation*}
    Q_{\theta,R}=Q_{\theta,R}^p
    := \left\{
    (t,x,v) \in \R^{1+2d} :
    t \in [-\theta R^p,0),
    \ x-tv \in B_{\theta R^{1+p}}(0),
    \ v \in B_R(0)
    \right\},
  \end{equation*}
  where $B_\rho(0)\subset\R^d$ denotes the Euclidean ball of radius $\rho$.
\end{definition}

Using the kinetic group law, the cylinder centred at $z_0=(t_0,x_0,v_0)$ is
\begin{align*}
   & Q_{\theta,R}(z_0)
  := \bigl\{ z\in\R^{1+2d} : z_0^{-1}\circ z \in Q_{\theta,R}\bigr\} \\
   & \;\;\; = \Bigl\{(t,x,v):\ t\in[t_0-\theta R^p,t_0),\;
  x-x_0-(t-t_0)v \in B_{\theta R^{1+p}}(0),\;
  v-v_0\in B_R(0)\Bigr\}.
\end{align*}

Note that $Q_{\theta,R}=\delta_R(Q_{\theta,1})$ and that the parameter $\theta$ plays the
role of an \emph{intrinsic} factor, allowing the time and transport scales to
adapt to the local size of the solution.

\subsection{Weak solutions}
We follow the standard energy notion of weak solutions used for Kolmogorov-type
kinetic equations \cite{dietert2025criticaltrajectorieskineticgeometry}, adapted
here to the nonlinear $p$-Laplace diffusion in $v$.

Let $I\subset\R$ be an open interval, and let $\Omega_x,\Omega_v\subset\R^d$ be open sets.
We write $\Omega_I := I\times\Omega_x\times\Omega_v$.

\begin{definition}[Weak (sub-)solutions]\label{def:weak-solution}
  Let $p\in(1,\infty)$. A measurable function
  \[
    f \in \L^\infty\bigl(I;\L^2(\Omega_x\times\Omega_v)\bigr)
    \cap \L^p(\Omega_I)
    \cap \L^p\bigl(I\times\Omega_x;\H^{1,p}(\Omega_v)\bigr)
  \]
  is called a \emph{weak (sub-)solution} to \eqref{eq:nonlinear-laplace} on
  $\Omega_I$ if for every test function $\varphi\in \C_c^\infty(\Omega_I)$ with $\varphi\ge 0$ one has
  \begin{equation*}
    \int_{\Omega_I}
    \Bigl(
    - f\,(\partial_t+v\cdot\nabla_x)\varphi
    + A(t,x,v,f,\nabla_vf)\cdot\nabla_v\varphi
    \Bigr) \dx (t,x,v)
    =(\le) 0.
  \end{equation*}
  In particular, $f$ is a weak solution if the identity holds for all
  $\varphi\in \C_c^\infty(\Omega_I)$.
\end{definition}

\begin{remark}\label{rem:admissible-tests}
  In the energy estimate below, we test the equation with functions of the form
  $f\,\chi^p\,\eta$, which are generally not sufficiently smooth. As in degenerate parabolic theory, this can
  be justified by a regularisation argument and a density/approximation
  procedure. Since our focus is on the algebraic structure of the estimate, we
  present the computations formally; they can be made rigorous within the
  framework of Definition~\ref{def:weak-solution}. We refer the interested
  reader to \cite[Appendix~A]{dietert2025criticaltrajectorieskineticgeometry}
  for details.
\end{remark}

\subsection{A localised gain of integrability for subsolutions}

\begin{lemma}\label{lem:gainsubloc}
  Let
  \[
    p\in\Bigl(2-\frac{2}{3d+2},\,2+\frac{2}{d}\Bigr), \qquad
    q := p\,\frac{4d+2}{d(p+2)}, \qquad
    r := p\,\frac{4d+2}{p(3d+2)-2d}.
  \]
  Fix $\theta>0$, $0<R_1<R_2$ and $(t_0,x_0,v_0) \in \R^{1+2d}$.  Assume \eqref{eq:a-bounds}, and let $f\ge 0$ be a weak
  subsolution in $Q_{\theta,R_2} := Q_{\theta,R_2}(t_0,x_0,v_0)$ of the equation \eqref{eq:nonlinear-laplace} with $\nabla_v f\in \L^p(Q_{\theta,R_2})$
  and $f\in \L^p(Q_{\theta,R_2})\cap \L^2(Q_{\theta,R_2})$.  Define
  \[
    \Gamma_t:=\frac{1}{\theta\,(R_2^{p}-R_1^{p})},
    \qquad
    \Gamma_v:=\frac{R_2^{p}}{R_2^{1+p}-R_1^{1+p}}+\frac{1}{R_2-R_1}.
  \]
  Then there exists $C=C(p,d,\lambda,\Lambda)>0$ such that
  \begin{multline*}
    \|f\|_{\L^q(Q_{\theta,R_1})}
    \le C\Bigl(
    \|\nabla_v f\|_{\L^p(Q_{\theta,R_2})} +\bigl(1+\Gamma_v\,|Q_{\theta,R_2}|^{\frac1r-\frac{p-1}{p}}\bigr)
    \|\nabla_v f\|_{\L^p(Q_{\theta,R_2})}^{p-1}\\
    +\Gamma_v\,\|f\|_{\L^p(Q_{\theta,R_2})}
    +\Gamma_t\,|Q_{\theta,R_2}|^{\frac1r-\frac1p}\,\|f\|_{\L^p(Q_{\theta,R_2})}
    \Bigr)
  \end{multline*}
  and for \(p \ge 2\) also
  \begin{multline*}
    \|f\|_{\L^q(Q_{\theta,R_1})}
    \le C\Bigl(
    \|\nabla_v f\|_{\L^p(Q_{\theta,R_2})} +\bigl(1+\Gamma_v\,|Q_{\theta,R_2}|^{\frac1r-\frac{p-1}{p}}\bigr)
    \|\nabla_v f\|_{\L^p(Q_{\theta,R_2})}^{p-1}\\
    +\Gamma_v\,\|f\|_{\L^p(Q_{\theta,R_2})}
    +\Gamma_t\,|Q_{\theta,R_2}|^{\frac1r-\frac12}\,\|f\|_{\L^2(Q_{\theta,R_2})}
    \Bigr).
  \end{multline*}
\end{lemma}

\begin{proof}
  Recall the notion $B_\rho:=\{y\in\R^d:\ |y|<\rho\}$ from
  Definition~\ref{def:kinetic-cylinder}. Since kinetic translations transform \eqref{eq:nonlinear-laplace} into a structurally equivalent equation, we may assume $(t_0,x_0,v_0)=0$.

  Choose smooth cutoffs
  \[
    \eta\in \C^\infty((-\theta R_2^p,\infty)),\qquad
    \zeta\in \C_c^\infty(\R^d),\qquad
    \varphi\in \C_c^\infty(\R^d)
  \]
  taking values in $[0,1]$ such that
  \[
    \eta\equiv 1 \text{ on }[-\theta R_1^p,\infty),\quad
    \zeta\equiv 1 \text{ on }B_{\theta R_1^{1+p}},\quad
    \varphi\equiv 1 \text{ on }B_{R_1},
  \]
  and
  \[
    \supp\eta\subset(-\theta R_2^p,\infty),\quad
    \supp\zeta\subset B_{\theta R_2^{1+p}},\quad
    \supp\varphi\subset B_{R_2}.
  \]
  Define the characteristic-adapted cutoff
  \[
    \chi(t,x,v):=\eta(t)\,\zeta(x-tv)\,\varphi(v).
  \]
  Then $\chi\equiv 1$ on $Q_{\theta,R_1}$ and $\supp\chi \cap ((-\infty,0) \times \R^{2d})\subset Q_{\theta,R_2}$. Moreover, $(\partial_t+v\cdot\nabla_x)\zeta(x-tv)=0$ implies
  \[
    (\partial_t+v\cdot\nabla_x)\chi=\eta'(t)\,\zeta(x-tv)\,\varphi(v).
  \]
  With standard choices of one-dimensional cutoffs, we can ensure
  \[
    \|\eta'\|_{\L^\infty}\le \frac{C}{\theta\,(R_2^p-R_1^p)},\qquad
    \|\nabla\zeta\|_{\L^\infty}\le \frac{C}{\theta\,(R_2^{1+p}-R_1^{1+p})},\qquad
    \|\nabla\varphi\|_{\L^\infty}\le \frac{C}{R_2-R_1},
  \]
  where $C=C(d)$. Hence
  \[
    \|(\partial_t+v\cdot\nabla_x)\chi\|_{\L^\infty}\le C\,\Gamma_t.
  \]
  Also,
  \[
    \nabla_v\chi
    =\eta(t)\Bigl(-t\,[\nabla\zeta](x-tv)\,\varphi(v)+\zeta(x-tv)\,[\nabla\varphi](v)\Bigr),
  \]
  and since $|t|\le \theta R_2^p$ on $\supp\chi$ we obtain
  \[
    \|\nabla_v\chi\|_{\L^\infty}
    \le C\left(\frac{R_2^p}{R_2^{1+p}-R_1^{1+p}}+\frac{1}{R_2-R_1}\right)
    = C\,\Gamma_v,
  \]
  so the $\theta$-dependence cancels in $\Gamma_v$.

  \medskip

  Set $g:=f\chi$. Since $\chi\equiv 1$ on $Q_{\theta,R_1}$,
  \[
    \|f\|_{\L^q(Q_{\theta,R_1})}\le \|g\|_{\L^q((-\infty,0) \times \R^{2d})}.
  \]
  In the sense of weak subsolutions, we have
  \begin{align*}
    (\partial_t+v\cdot\nabla_x)g
     & =(\partial_t+v\cdot\nabla_x)(f\chi)
    \le \chi\,\nabla_v\cdot A + f\,(\partial_t+v\cdot\nabla_x)\chi                                 \\
     & =\nabla_v\cdot(\chi A)+\Bigl[f\,(\partial_t+v\cdot\nabla_x)\chi - A\cdot\nabla_v\chi\Bigr].
  \end{align*}
  Thus $g$ satisfies
  \[
    (\partial_t+v\cdot\nabla_x)g \le \nabla_v\cdot S_0 + S_1,
    \qquad
    S_0:=\chi A,\quad S_1:=f\,(\partial_t+v\cdot\nabla_x)\chi - A\cdot\nabla_v\chi.
  \]

  We apply the global estimate of Theorem~\ref{thm:gainsubglobal} to $g$ with
  the exponents $q,r$ and $T = 0$ to obtain
  \[
    \|g\|_{\L^q((-\infty, 0) \times \R^{2d})}
    \le C\Bigl(\|\nabla_v g\|_{\L^p((-\infty, 0) \times \R^{2d})}+\|S_0\|_{\L^{\frac{p}{p-1}}((-\infty, 0) \times \R^{2d})}+\|S_1\|_{\L^r((-\infty, 0) \times \R^{2d})}\Bigr),
  \]
  for a constant $C=C(p,d)$.  Since $\supp\chi\subset Q_{\theta,R_2}$,
  $\nabla_v g=\chi\,\nabla_v f + f\,\nabla_v\chi$ implies
  \begin{align*}
    \|\nabla_v g\|_{\L^p((-\infty, 0) \times \R^{2d})}
     & \le \|\nabla_v f\|_{\L^p(Q_{\theta,R_2})}+\|\nabla_v\chi\|_{\L^\infty}\,\|f\|_{\L^p(Q_{\theta,R_2})} \\
     & \le \|\nabla_v f\|_{\L^p(Q_{\theta,R_2})}+C\,\Gamma_v\,\|f\|_{\L^p(Q_{\theta,R_2})}.
  \end{align*}
  Moreover, $|S_0|=\chi|A|\le \Lambda \,|\nabla_v f|^{p-1}$ gives
  \[
    \|S_0\|_{\L^{\frac{p}{p-1}}(Q_{\theta,R_2})}
    \le \Lambda\|\nabla_v f\|_{\L^p(Q_{\theta,R_2})}^{p-1}.
  \]

  Write $S_1=S_{1,t}+S_{1,v}$ with $S_{1,t}:=f\,(\partial_t+v\cdot\nabla_x)\chi$ and $S_{1,v}:=-A\cdot\nabla_v\chi$.
  In the present range of $p$, we always have \(r \le p\).  Hence
  H\"older's inequality on the finite-measure set $Q_{\theta,R_2}$ yields
  \[
    \|f\|_{\L^r(Q_{\theta,R_2})}\le |Q_{\theta,R_2}|^{\frac1r-\frac1p}\,\|f\|_{\L^p(Q_{\theta,R_2})}.
  \]
  For \(p \ge 2\), we also have
  \[
    \|f\|_{\L^r(Q_{\theta,R_2})}\le |Q_{\theta,R_2}|^{\frac1r-\frac12}\,\|f\|_{\L^2(Q_{\theta,R_2})}.
  \]
  For the \(A\)-term, we find
  \[
    \|A\|_{\L^r(Q_{\theta,R_2})}
    \le |Q_{\theta,R_2}|^{\frac1r-\frac{p-1}{p}}\,
    \|A\|_{\L^{\frac{p}{p-1}}(Q_{\theta,R_2})}
    \le \Lambda|Q_{\theta,R_2}|^{\frac1r-\frac{p-1}{p}}\,
    \|\nabla_v f\|_{\L^p(Q_{\theta,R_2})}^{p-1}.
  \]
  Therefore,
  \[
    \|S_{1,t}\|_{\L^r(Q_{\theta,R_2})}
    \le \|(\partial_t+v\cdot\nabla_x)\chi\|_{\L^\infty_{t,x,v}}\,\|f\|_{\L^r(Q_{\theta,R_2})}
    \le C\,\Gamma_t\,|Q_{\theta,R_2}|^{\frac1r-\frac12}\,\|f\|_{\L^2(Q_{\theta,R_2})},
  \]
  for $p \ge 2$,
  \[
    \|S_{1,t}\|_{\L^r(Q_{\theta,R_2})}
    \le  C
    \Gamma_t\,|Q_{\theta,R_2}|^{\frac1r-\frac1p}\,
    \|f\|_{\L^p(Q_{\theta,R_2})},
  \]
  for all $p$ in the given range, and
  \[
    \|S_{1,v}\|_{\L^r(Q_{\theta,R_2})}
    \le \|\nabla_v\chi\|_{\L^\infty_{t,x,v}}\,\|A\|_{\L^r(Q_{\theta,R_2})}
    \le C\,\Gamma_v\,|Q_{\theta,R_2}|^{\frac1r-\frac{p-1}{p}}\,
    \|\nabla_v f\|_{\L^p(Q_{\theta,R_2})}^{p-1}.
  \]
  Combining these bounds gives the stated estimate.
\end{proof}

\subsection{Energy estimate}
We now prove the basic Caccioppoli-type estimate on nested kinetic $p$-cylinders.

\medskip
For $\theta,R>0$ and $t\in[-\theta R^p,0)$, we denote the time-slice of $Q_{\theta,R}$ at time $t$ by
\[
  D_{\theta,R}(t)
  := \bigl\{(x,v)\in\R^{2d}:\ x-tv\in B_{\theta R^{1+p}}(0),\ v\in B_R(0)\bigr\}.
\]
For a measurable function $g$ on $Q_{\theta,R}$ we write
\[
  \|g\|_{\L^\infty_t\L^2_{x,v}(Q_{\theta,R})}
  := \esssup_{t\in(-\theta R^p,0)}
  \left(\int_{D_{\theta,R}(t)} |g(t,x,v)|^2 \dx (x,v) \right)^{1/2}.
\]

\begin{lemma}[Energy estimate]\label{lem:energy}
  Let $p\in(1,\infty)$, $\theta>0$, and $0<R_1<R_2$. Assume \eqref{eq:a-bounds}, and let $f$ be a nonnegative weak
  subsolution of \eqref{eq:nonlinear-laplace} in an open
  set containing $\overline{Q_{\theta,R_2}}$.  Then there exists $C=C(d,p,\lambda,\Lambda)>0$ such
  that
  \begin{equation*}
    \|f\|_{\L^\infty_t \L^2_{x,v}(Q_{\theta,R_1})}^2
    + \|\nabla_v f\|_{\L^p(Q_{\theta,R_1})}^p
    \le
    \frac{C}{\theta\,(R_2-R_1)^p}\,\|f\|_{\L^2(Q_{\theta,R_2})}^2
    + \frac{C}{(R_2-R_1)^p}\,\|f\|_{\L^p(Q_{\theta,R_2})}^p.
  \end{equation*}
\end{lemma}

\begin{remark}
  The $\L^\infty_t\L^2_{x,v}$ bound is not needed for our approach to local boundedness. We decided to include it to motivate our definition of energy weak solutions and because it may be useful for future purposes.
\end{remark}

\begin{proof}[Proof of Lemma \ref{lem:energy}]
  Fix $0<R_1<R_2$ and write $Q_{R_i}:=Q_{\theta,R_i}$ for \(i=1,2\).

  \textbf{Step 1. Cutoffs adapted to the transport.}
  Choose $\chi_x\in \C_c^\infty(\R^d)$ and $\chi_v\in \C_c^\infty(\R^d)$ such that
  \[
    \begin{aligned}
       & 0\le \chi_x,\chi_v\le 1,                              \\
       & \chi_x\equiv 1\ \text{on }B_{\theta R_1^{1+p}},\qquad
      \supp\chi_x\subset B_{\theta R_2^{1+p}},                 \\
       & \chi_v\equiv 1\ \text{on }B_{R_1},\qquad
      \supp\chi_v\subset B_{R_2},
    \end{aligned}
  \]
  and
  \begin{equation}
    \label{eq:cutoff-derivatives}
    |\nabla \chi_x|\le \frac{C}{\theta\,(R_2^{1+p}-R_1^{1+p})},
    \qquad
    |\nabla \chi_v|\le \frac{C}{R_2-R_1}.
  \end{equation}
  Define $\chi:\R^{2d}\to[0,1]$ by $\chi(y,v):=\chi_x(y)\chi_v(v)$, and set the kinetic cutoff
  \[
    \chi_{R_2}(t,x,v) := \chi(x-tv,v), \qquad (t,x,v)\in\R^{1+2d}.
  \]
  By construction,
  \[
    (\partial_t+v\cdot\nabla_x)\chi_{R_2}(t,x,v)=0 \quad\text{for all }(t,x,v),
  \]
  and for the velocity derivative, we have
  \begin{equation*}
    \nabla_v \chi_{R_2}(t,x,v)
    = \bigl(\nabla_v \chi\bigr)(x-tv,v) - t\,\bigl(\nabla_y \chi\bigr)(x-tv,v).
  \end{equation*}
  Hence, using \eqref{eq:cutoff-derivatives} and $|t|\le \theta R_2^p$ on $Q_{R_2}$,
  \[
    |\nabla_v \chi_{R_2}|
    \le \frac{C}{R_2-R_1} + |t|\frac{C}{\theta\,(R_2^{1+p}-R_1^{1+p})}
    \le \frac{C}{R_2-R_1} + \theta R_2^p\frac{C}{\theta\,(R_2^{1+p}-R_1^{1+p})}.
  \]
  We estimate the last ratio using the elementary inequality
  \[
    R_2^{1+p}-R_1^{1+p}
    = R_2^{1+p}\Bigl(1-(R_1/R_2)^{1+p}\Bigr)
    \ge R_2^{1+p}\Bigl(1-R_1/R_2\Bigr)
    = R_2^p(R_2-R_1),
  \]
  which holds since $(R_1/R_2)^{1+p}\le R_1/R_2$.
  Therefore,
  \begin{equation}
    \label{eq:gradv-chi-bound}
    |\nabla_v \chi_{R_2}|\le \frac{C}{R_2-R_1}
    \quad\text{on }Q_{R_2}.
  \end{equation}

  \medskip
  Next, choose a time cutoff $\eta\in \C_c^\infty((-\theta R_2^p,0])$ such that
  \[
    0\le \eta\le 1,\qquad
    \eta\equiv 1\ \text{on }[-\theta R_1^p,0],\qquad
    \eta(-\theta R_2^p)=0,
  \]
  and
  \begin{equation}
    \label{eq:eta-derivative}
    |\dot\eta(t)|\le \frac{C}{\theta\,(R_2^p-R_1^p)}\le \frac{C}{\theta\,(R_2-R_1)^p}.
  \end{equation}
  The last inequality follows from
  \[
    R_2^p-R_1^p = (R_1+(R_2-R_1))^p-R_1^p \ge (R_2-R_1)^p,
  \]
  which holds for all $p\ge 1$ since, for every $b>0$, the function $a\mapsto (a+b)^p-a^p$ is increasing for
  $a\ge 0$ and equals $b^p$ at $a=0$.

  \textbf{Step 2. Testing the equation.}
  Set
  \[
    \psi(t,x,v) := \chi_{R_2}(t,x,v)^p\,\eta(t).
  \]
  Formally (and rigorously by the regularisation procedure explained in
  Remark~\ref{rem:admissible-tests}), we test \eqref{eq:nonlinear-laplace}
  against $f\psi$ and integrate over $Q_{R_2}$:
  \begin{equation}
    \label{eq:test-identity}
    \int_{Q_{R_2}} (\partial_t+v\cdot\nabla_x)f\; f\psi \dx (t,x,v)
    - \int_{Q_{R_2}}\nabla_v\cdot A\big(t,x,v,f,\nabla_v f\big)\; f\psi\dx (t,x,v) \le 0.
  \end{equation}

  \textbf{Step 3. The transport contribution.}
  We write
  \[
    \int_{Q_{R_2}} (\partial_t+v\cdot\nabla_x)f\; f\psi \dx (t,x,v)
    = \frac12 \int_{Q_{R_2}} (\partial_t+v\cdot\nabla_x)(f^2)\,\psi \dx (t,x,v).
  \]
  Since $\psi$ is compactly supported in $(x,v)$, integrating by parts with respect to the
  transport vector field yields
  \begin{align*}
     & \frac12 \int_{Q_{R_2}} (\partial_t+v\cdot\nabla_x)(f^2)\,\psi \dx (t,x,v)                \\
     & = \frac12\int_{\R^{2d}} f^2(t,x,v)\,\psi(t,x,v) \dx (x,v) \Bigg|_{t=-\theta R_2^p}^{t=0}
    - \frac12 \int_{Q_{R_2}} f^2\,(\partial_t+v\cdot\nabla_x)\psi\dx (t,x,v).
  \end{align*}
  Moreover, $(\partial_t+v\cdot\nabla_x)\chi_{R_2}^p=0$, hence
  $(\partial_t+v\cdot\nabla_x)\psi = \chi_{R_2}^p\,\dot\eta(t)$. Using also $\eta(-\theta R_2^p)=0$
  we obtain
  \begin{align}
    \label{eq:transport-term}
     & \int_{Q_{R_2}} (\partial_t+v\cdot\nabla_x)f\; f\psi\dx (t,x,v) \\
     & = \frac12\int_{\R^{2d}} f(0,x,v)^2\,\chi(x,v)^p \dx (x,v)
    - \frac12 \int_{Q_{R_2}} f^2\,\chi_{R_2}^p\,\dot\eta(t)\dx (t,x,v). \nonumber
  \end{align}

  \textbf{Step 4. The diffusion contribution.}
  Integration by parts in $v$ (no boundary term since $\chi_{R_2}$ is compactly supported in $v$)
  gives
  \begin{multline}\label{eq:diffusion-term-raw}
    - \int_{Q_{R_2}}\nabla_v\cdot A\big(t,x,v,f,\nabla_v f\big)\; f\psi \dx (t,x,v) \\
    = \int_{Q_{R_2}} A\big(t,x,v,f,\nabla_v f\big)\cdot (\nabla_v f) \,\chi_{R_2}^p\,\eta \dx (t,x,v)\\
    + p\int_{Q_{R_2}} A\big(t,x,v,f,\nabla_v f\big) \cdot(\nabla_v\chi_{R_2})\,f\,\chi_{R_2}^{p-1}\,\eta \dx (t,x,v).
  \end{multline}
  The first term gives control of \(\lambda |\nabla_vf|^p\) by \eqref{eq:a-bounds}.
  By the upper bound in \eqref{eq:a-bounds}, we estimate the mixed term by
  Young's inequality as
  \[
    p\,|A(t,x,v,f,\nabla_vf)| \chi_{R_2}^{p-1}\,|f|\,|\nabla_v\chi_{R_2}|
    \le \frac{\lambda}{2} |\nabla_v f|^p\chi_{R_2}^p
    + C(p,\lambda,\Lambda)\,|f|^p\,|\nabla_v\chi_{R_2}|^p.
  \]
  Inserting this into \eqref{eq:diffusion-term-raw}, we obtain
  \begin{multline}\label{eq:diffusion-term}
    - \int_{Q_{R_2}}\nabla_v\cdot A\Bigl(t,x,v,f,\nabla_v f\Bigr)\; f\psi  \dx (t,x,v) \\
    \ge \frac{\lambda}{2} \int_{Q_{R_2}} |\nabla_v f|^{p}\,\chi_{R_2}^p\,\eta \dx (t,x,v)
    - C(p,\lambda,\Lambda) \int_{Q_{R_2}} |f|^p\,|\nabla_v\chi_{R_2}|^p\,\eta \dx (t,x,v).
  \end{multline}

  \textbf{Step 5. Conclusion for the gradient term and for a time slice.}
  Combining \eqref{eq:test-identity}, \eqref{eq:transport-term} and \eqref{eq:diffusion-term}
  yields
  \begin{multline}\label{eq:energy-intermediate}
    \frac12\int_{\R^{2d}} f(0,x,v)^2\,\chi(x,v)^p \dx (x,v)
    + \frac{\lambda}{2} \int_{Q_{R_2}} |\nabla_v f|^{p}\,\chi_{R_2}^p\,\eta \dx (t,x,v) \\
    \le
    \frac12 \int_{Q_{R_2}} f^2\,\chi_{R_2}^p\,|\dot\eta(t)|\dx (t,x,v)
    + C(p,\lambda,\Lambda)\int_{Q_{R_2}} |f|^p\,|\nabla_v\chi_{R_2}|^p\,\eta\dx (t,x,v) .
  \end{multline}
  Using \eqref{eq:eta-derivative} and \eqref{eq:gradv-chi-bound}, and $\chi_{R_2}^p\eta\le 1$,
  we estimate the right-hand side by
  \begin{align*}
    \frac12 \int_{Q_{R_2}} f^2\,\chi_{R_2}^p\,|\dot\eta(t)| \dx (t,x,v)
     & \le \frac{C}{\theta\,(R_2-R_1)^p}\,\int_{Q_{R_2}} |f|^2\dx (t,x,v), \\
    \int_{Q_{R_2}} |f|^p\,|\nabla_v\chi_{R_2}|^p\,\eta\dx (t,x,v)
     & \le \frac{C}{(R_2-R_1)^p}\,\int_{Q_{R_2}} |f|^p \dx (t,x,v).
  \end{align*}
  Since $\chi_{R_2}\equiv 1$ and $\eta\equiv 1$ on $Q_{R_1}$, we deduce from \eqref{eq:energy-intermediate},
  \begin{align*}
     & \int_{D_{\theta,R_1}(0)} |f(0,x,v)|^2 \dx (x,v)
    + \int_{Q_{R_1}} |\nabla_v f|^p \dx(t,x,v)         \\
     & \le
    \frac{C}{\theta\,(R_2-R_1)^p}\,\int_{Q_{R_2}} |f|^2\dx(t,x,v)
    + \frac{C}{(R_2-R_1)^p}\,\int_{Q_{R_2}} |f|^p\dx(t,x,v).
  \end{align*}

  \textbf{Step 6. Upgrading to an $\L^\infty_t\L^2_{x,v}$ bound.}
  Fix $\tau\in(-\theta R_1^p,0)$. We repeat Steps~2--5 with the same spatial cutoff
  $\chi_{R_2}(t,x,v)=\chi(x-tv,v)$ but only integrate in time until \(\tau\).
  Concretely, let $\eta$ be the cutoff from Step~1 satisfying
  $\eta(-\theta R_2^p)=0$, $\eta\equiv 1$ on $[-\theta R_1^p,0]$, and
  $|\dot\eta|\lesssim (\theta(R_2-R_1)^p)^{-1}$, and define formally
  \[
    \eta_\tau(t):=\eta(t)\,\mathds{1}_{\{t\le \tau\}}.
  \]
  Then $\eta_\tau\equiv 1$ on $[-\theta R_1^p,\tau]$, $\eta_\tau(-\theta R_2^p)=0$, and
  $\eta_\tau(0)=0$. Testing the equation with $f\,\chi_{R_2}^p\,\eta_\tau$ and repeating
  the computations of Steps~3--5 (with the same estimates on $\dot\eta$ and
  $\nabla_v\chi_{R_2}$) yields the analogue of \eqref{eq:energy-intermediate}, but now
  with the boundary term at time $\tau$
  \begin{multline*}
    \frac12 \int_{\R^{2d}} f(\tau,x,v)^2\,\chi(x-\tau v,v)^p\dx (x,v)
    + \frac{\lambda}{2} \int_{Q_{\theta,R_2}\cap\{t<\tau\}} |\nabla_v f|^p\,\chi_{R_2}^p\dx(t,x,v)
    \\
    \le
    \frac12 \int_{Q_{\theta,R_2}\cap\{t<\tau\}} f^2\,\chi_{R_2}^p\,|\dot\eta|\dx(t,x,v)
    + C \int_{Q_{\theta,R_2}\cap\{t<\tau\}} |f|^p\,|\nabla_v\chi_{R_2}|^p\dx(t,x,v) .
  \end{multline*}
  Using $|\dot\eta|\lesssim(\theta(R_2-R_1)^p)^{-1}$, $|\nabla_v\chi_{R_2}|^p\lesssim (R_2-R_1)^{-p}$
  and enlarging the integration domain on the right-hand side, we obtain
  \[
    \int_{\R^{2d}} f(\tau,x,v)^2\,\chi(x-\tau v,v)^p\dx (x,v)
    \lesssim
    \frac{1}{\theta(R_2-R_1)^p}\|f\|_{\L^2(Q_{\theta,R_2})}^2
    + \frac{1}{(R_2-R_1)^p}\|f\|_{\L^p(Q_{\theta,R_2})}^p.
  \]
  Since $\chi\equiv 1$ on $B_{\theta R_1^{1+p}}\times B_{R_1}$, we have
  $\chi(x-\tau v,v)\equiv 1$ on $D_{\theta,R_1}(\tau)$, hence
  \[
    \int_{D_{\theta,R_1}(\tau)} |f(\tau,x,v)|^2\dx (x,v)
    \lesssim
    \frac{1}{\theta(R_2-R_1)^p}\|f\|_{\L^2(Q_{\theta,R_2})}^2
    + \frac{1}{(R_2-R_1)^p}\|f\|_{\L^p(Q_{\theta,R_2})}^p.
  \]
  Taking the supremum over $\tau\in(-\theta R_1^p,0)$ yields the desired
  $\|f\|_{\L^\infty_t\L^2_{x,v}(Q_{\theta,R_1})}$ bound.
\end{proof}

\section{Local boundedness of subsolutions}
\label{sec:iteration}

Using the gain of integrability, we now prove the local boundedness of
nonnegative subsolutions via De Giorgi iteration on $p$-adapted
kinetic cylinders. A classical treatment of local boundedness can be found in
DiBenedetto~\cite{dibenedetto-1993-degenerate}, where the smallness condition
is formulated in terms of low integrability.  For this presentation, we did not
optimise this aspect and sought a simple proof inspired by
Lindfors~\cite{lindfors-2016-obstacle-laplacian}.

We consider first the case $p \ge 2$, where we iterate the truncated $\L^p$ norm.

\begin{theorem}\label{thm:DeGiorgi-localbdd}
  Let $p\in\Bigl[2,\,2+\frac{2}{d}\Bigr)$ and assume \eqref{eq:a-bounds}.  Let
  $f\ge 0$ be a weak subsolution of \eqref{eq:nonlinear-laplace} in an open set
  containing $\overline{Q_{\Theta,2R}(z_0)}$. Then there exists a constant
  $\varepsilon_0=\varepsilon_0(d,p,\lambda,\Lambda)\in(0,1)$ such that the following holds.

  \medskip
  \noindent\textbf{Intrinsic boundedness principle.}
  Fix any level $K>0$ and set the intrinsic factor
  \begin{equation}
    \Theta := K^{2-p}.
    \label{eq:Theta-intrinsic}
  \end{equation}
  If $f$ satisfies the normalised smallness condition
  \begin{equation}\label{eq:S}
    \fint_{Q_{\Theta,2R}(z_0)}\Bigl(\frac{f}{K}\Bigr)^p \dx (t,x,v)
    \le\varepsilon_0.
  \end{equation}
  then
  \begin{equation}\label{eq:B}
    \esssup_{Q_{\Theta,R}(z_0)} f \le K.
  \end{equation}
  In particular, nonnegative weak subsolutions are locally bounded.
\end{theorem}

\begin{proof}
  The proof is a De Giorgi iteration combining (i) the energy estimate
  (Lemma~\ref{lem:energy}) on nested kinetic $p$-cylinders, and (ii) the
  localised gain of integrability (Lemma~\ref{lem:gainsubloc}).  The key
  nonlinear bookkeeping point is precisely the intrinsic relation
  \eqref{eq:Theta-intrinsic}, which balances the $\L^2$ term coming from the
  time cutoff against the level increments in the De Giorgi truncations. We
  recall
  \[
    q := p\,\frac{4d+2}{d(p+2)},\qquad
    r := p\,\frac{4d+2}{p(3d+2)-2d}.
  \]

  \textbf{Step 0. Normalisation (translation and scaling).}
  Since kinetic translations transform \eqref{eq:nonlinear-laplace} into a structurally equivalent equation, we may assume $(t_0,x_0,v_0)=0$.
  Correspondingly, employing the scaling $\delta_R$, we reduce to the case $R=1$.
  Thus we work on $Q_{\Theta,2}$ and aim to prove
  \[
    \esssup_{Q_{\Theta,1}} f \le K
    \quad\text{under \eqref{eq:S}.}
  \]

  Now, we apply the intrinsic scaling
  \[
    u(t,x,v):=\frac{1}{K}\,f\bigl(\Theta t,\ \Theta x,\ v\bigr),
    \qquad \Theta = K^{2-p}.
  \]
  By the compatibility of \eqref{eq:Theta-intrinsic} with the structural properties of \eqref{eq:nonlinear-laplace}, the function $u$ is a weak subsolution of the structurally equivalent equation
  \begin{equation} \label{TildeAEquation}
    (\partial_t + v \cdot \nabla_x)u
    - \nabla_{v} \cdot \tilde{A}(t,x,v,u,\nabla_v u) = 0\quad\text{in}\;\,Q_{1,2},
  \end{equation}
  where $\tilde{A}(t,x,v,\eta,\xi)=K^{1-p}A(\Theta t,\Theta x,v,K\eta,K\xi)$.
  Moreover, since $|Q_{\Theta,2}|=\Theta^{d+1}|Q_{1,2}|$, a change of variables gives
  \[
    \fint_{Q_{1,2}} u^p\dx (t,x,v)
    =
    \fint_{Q_{\Theta,2}} \Bigl(\frac{f}{K}\Bigr)^p\dx (t,x,v).
  \]
  Hence \eqref{eq:S} becomes
  \begin{equation}\label{eq:S0}
    \fint_{Q_{1,2}} u^p\dx (t,x,v)
    \le \varepsilon_0.
  \end{equation}
  It suffices to prove the normalised claim
  \begin{equation}\label{eq:B0}
    \esssup_{Q_{1,1}} u \le 1.
  \end{equation}

  \textbf{Step 1. De Giorgi sequences of levels and radii.}
  Define radii and levels
  \[
    R_n := 1+2^{-n},\qquad k_n := 1-2^{-n},\qquad n=0,1,2,\dots
  \]
  so $R_0=2$, $R_n\downarrow 1$, and $k_0=0$, $k_n\uparrow 1$.
  Let
  \[
    w_n := (u-k_n)_+.
  \]
  Standard truncation arguments show that $w_n$ is also a weak subsolution.
  Define the energy quantities
  \[
    M_n := \int_{Q_{1,R_n}} w_n^p\dx (t,x,v).
  \]
  We will show that $M_n\to 0$ provided that $\varepsilon_0$ is sufficiently small. Since $w_n\downarrow (u-1)_+$ pointwise, this implies $(u-1)_+=0$
  a.e.\ in $Q_{1,1}$, i.e.\ \eqref{eq:B0}.

  \textbf{Step 2. Two basic truncation estimates.}
  We start with the level-set estimate.
  Let
  \[
    E_{n+1}:=\{u>k_{n+1}\}\cap Q_{1,R_{n+1}}.
  \]
  On $E_{n+1}$ we have $w_n=u-k_n \ge k_{n+1}-k_n = 2^{-(n+1)}$, hence
  \begin{equation}
    |E_{n+1}|
    \le 2^{p(n+1)} \int_{Q_{1,R_{n+1}}} w_n^p\dx (t,x,v)
    \le 2^{p(n+1)} M_n.
    \label{eq:level-set}
  \end{equation}

  Next, we provide an $\L^2$ bound for $w_n$ in terms of $M_{n-1}$, where we
  use the additional assumption \(p\ge 2\).  For $n\ge1$, on $\{w_n>0\}=\{u>k_n\}$
  we have $w_{n-1}=u-k_{n-1}\ge k_n-k_{n-1}=2^{-n}$, hence
  $\{w_n>0\}\subset\{w_{n-1}\ge 2^{-n}\}$ and
  \begin{equation*}
    |\{w_n>0\}\cap Q_{1,R_{n-1}}|
    \le 2^{np}\int_{Q_{1,R_{n-1}}} w_{n-1}^p\dx (t,x,v)
    =2^{np}M_{n-1}.
  \end{equation*}
  By H\"older's inequality on $Q_{1,R_{n-1}}$ and using $w_n\le w_{n-1}$,
  \begin{multline}\label{eq:L2-vs-A}
    \int_{Q_{1,R_{n-1}}} w_n^2\dx (t,x,v) \\
    \le
    \Bigl(\int_{Q_{1,R_{n-1}}} w_n^p\dx (t,x,v)\Bigr)^{2/p}
    |\{w_n>0\}\cap Q_{1,R_{n-1}}|^{1-2/p}
    \le 2^{n(p-2)}M_{n-1}
  \end{multline}
  for $p \ge 2$.

  \textbf{Step 3. Gradient bound from the energy estimate.}
  Apply Lemma~\ref{lem:energy} to $w_n$ in the nested cylinders
  $Q_{1,R_{n-1}}\supset Q_{1,R_n}$ (take $R_1=R_n$, $R_2=R_{n-1}$, $\theta=1$):
  \begin{equation*}
    \|\nabla_v w_n\|_{\L^p(Q_{1,R_n})}^p
    \le
    \frac{C}{(R_{n-1}-R_n)^p}\,\|w_n\|_{\L^2(Q_{1,R_{n-1}})}^2
    +\frac{C}{(R_{n-1}-R_n)^p}\,\|w_n\|_{\L^p(Q_{1,R_{n-1}})}^p.
  \end{equation*}
  Since $R_{n-1}-R_n = 2^{-n}$, using $w_n\le w_{n-1}$ and \eqref{eq:L2-vs-A} yields
  \begin{equation*}
    \|\nabla_v w_n\|_{\L^p(Q_{1,R_n})}^p
    \le
    C\,2^{np}\Bigl(2^{n(p-2)}M_{n-1}+M_{n-1}\Bigr)
    =
    C\,2^{np}\bigl(1+2^{n(p-2)}\bigr)M_{n-1}.
  \end{equation*}

  \textbf{Step 4. $\L^q$ estimate from the gain of integrability.}
  Apply the gain-of-integrability to the subsolution $w_n$ with
  $R_2=R_n$, $R_1=R_{n+1}$, and $\theta=1$.
  Since $R_n\in[1,2]$, the cutoff costs satisfy
  \begin{equation*}
    \Gamma_t \lesssim 2^n,\qquad \Gamma_v\lesssim 2^n,
  \end{equation*}
  and $|Q_{1,R_n}|\simeq 1$ uniformly in $n$.
  Thus,
  \begin{equation}
    \|w_n\|_{\L^q(Q_{1,R_{n+1}})}
    \le C\Bigl[
      \|\nabla_v w_n\|_{\L^p(Q_{1,R_n})}
      +2^n\|\nabla_v w_n\|_{\L^p(Q_{1,R_n})}^{p-1}
      +2^n\|w_n\|_{\L^p(Q_{1,R_n})}
      \Bigr].
    \label{eq:Lq-cor}
  \end{equation}
  We continue by estimating each term:
  \begin{align}
    \|\nabla_v w_n\|_{\L^p(Q_{1,R_n})}
     & \le C\,2^n\bigl(1+2^{n(p-2)}\bigr)^{1/p}\,M_{n-1}^{1/p},
    \label{eq:term1}                                                           \\
    \|\nabla_v w_n\|_{\L^p(Q_{1,R_n})}^{p-1}
     & \le C\,2^{n(p-1)}\bigl(1+2^{n(p-2)}\bigr)^{(p-1)/p}\,M_{n-1}^{(p-1)/p},
    \notag                                                                     \\
    \|w_n\|_{\L^p(Q_{1,R_n})}
     & \le M_{n-1}^{1/p}.
    \label{eq:term3}
  \end{align}
  Plugging \eqref{eq:term1}--\eqref{eq:term3} into \eqref{eq:Lq-cor} yields
  \begin{multline}\label{eq:Lq-bound}
    \|w_n\|_{\L^q(Q_{1,R_{n+1}})}
    \le C\Bigl[
      2^n(1+2^{n(p-2)})^{1/p}M_{n-1}^{1/p}
      +2^{np}(1+2^{n(p-2)})^{(p-1)/p}M_{n-1}^{(p-1)/p}
      \Bigr].
  \end{multline}

  \textbf{Step 5. De Giorgi recursion for $M_n$.}
  Set
  \begin{equation*}
    \delta := 1-\frac{p}{q} = \frac{2(d+1)-dp}{4d+2} >0
    \qquad(\text{since }p<2+2/d).
  \end{equation*}
  Since $w_{n+1}\le w_n$ and $w_{n+1}>0$, i.e.\ $ u>k_{n+1}$ on the good set $E_{n+1}$,
  \[
    M_{n+1}
    =\int_{Q_{1,R_{n+1}}} w_{n+1}^p \dx (t,x,v)
    \le \int_{Q_{1,R_{n+1}}} w_n^p\,\indicator_{E_{n+1}} \dx (t,x,v).
  \]
  By H\"older's inequality with exponents $\frac{q}{p}$ and $\bigl(\frac{q}{p}\bigr)'$,
  \begin{equation*}
    M_{n+1}
    \le \|w_n\|_{\L^q(Q_{1,R_{n+1}})}^p\,|E_{n+1}|^{1-p/q}
    =
    \|w_n\|_{\L^q(Q_{1,R_{n+1}})}^p\,|E_{n+1}|^{\delta}.
  \end{equation*}
  Using \eqref{eq:level-set} and $M_n\le M_{n-1}$,
  \begin{equation*}
    |E_{n+1}|^\delta
    \le \bigl(2^{p(n+1)}M_n\bigr)^\delta
    \le 2^{p(n+1)\delta}M_{n-1}^{\delta}.
  \end{equation*}
  Hence
  \begin{equation}
    M_{n+1}\le 2^{p(n+1)\delta}\,\|w_n\|_{\L^q(Q_{1,R_{n+1}})}^p\,M_{n-1}^\delta.
    \label{eq:An-rec-1}
  \end{equation}
  Raising \eqref{eq:Lq-bound} to the power $p$ and using $(a+b)^p\le C(a^p+b^p)$
  gives
  \begin{equation*}
    \|w_n\|_{\L^q(Q_{1,R_{n+1}})}^p
    \le C\Bigl[
      2^{np}\bigl(1+2^{n(p-2)}\bigr)M_{n-1}
      +2^{np^2}\bigl(1+2^{n(p-2)}\bigr)^{p-1}M_{n-1}^{p-1}
      \Bigr].
  \end{equation*}
  Plugging this into \eqref{eq:An-rec-1} and absorbing harmless powers of $2$ into
  $2^{\gamma n}$ yields
  \begin{equation}
    M_{n+1}
    \le C\,2^{\gamma n}\Bigl[
      M_{n-1}^{1+\delta}
      +M_{n-1}^{p-1+\delta}
      \Bigr],
    \qquad n\ge 1,
    \label{eq:An-rec-raw}
  \end{equation}
  for some $\gamma=\gamma(d,p)$.

  Assuming (by choosing $\varepsilon_0\le |Q_{1,2}|^{-1}$) that $M_0\le 1$, we have $M_n\le 1$ for all $n$,
  and each exponent on $M_{n-1}$ in \eqref{eq:An-rec-raw} is $\ge 1+\delta$ (by the
  constraint \(p\ge2\)).  Thus, for $M_{n-1}\in(0,1]$,
  \begin{equation*}
    M_{n+1} \le C\,2^{\gamma n}\,M_{n-1}^{1+\delta},\qquad n\ge 1.
  \end{equation*}
  Equivalently, shifting indices,
  \begin{equation}
    M_{m+2} \le C\,2^{\gamma m}\,M_{m}^{1+\delta},\qquad m\ge 0.
    \label{eq:Am-rec}
  \end{equation}

  \textbf{Step 6. Convergence of the iteration.}
  Let $B_m:=M_{2m}$ and $C_m:=M_{2m+1}$. From \eqref{eq:Am-rec},
  \[
    B_{m+1}\le C\,2^{2\gamma m}\,B_m^{1+\delta},
    \qquad
    C_{m+1}\le C\,2^{\gamma(2m+1)}\,C_m^{1+\delta}.
  \]
  Both are of the form $Y_{m+1}\le C_1 b^m Y_m^{1+\delta}$ with $b>1$.  A standard
  fast-convergence lemma yields a $\delta_0=\delta_0(C_1,b,\delta)>0$ such that if $Y_0\le\delta_0$, then
  $Y_m\to 0$.  Thus, if $M_0\le\delta_0$, then $M_n\to 0$.

  Choose $\varepsilon_0$ in \eqref{eq:S0} so small that
  \[
    M_0=\int_{Q_{1,2}} u^p\dx (t,x,v)
    \le |Q_{1,2}|\fint_{Q_{1,2}} u^p\dx (t,x,v)
    \le |Q_{1,2}|\varepsilon_0 \le \delta_0.
  \]
  Then $M_n\to 0$ and, since $w_n\downarrow (u-1)_+$ pointwise on $Q_{1,1}$,
  \[
    0=\lim_{n\to\infty}M_n
    = \int_{Q_{1,1}} (u-1)_+^p\dx (t,x,v),
  \]
  so $(u-1)_+=0$ a.e.\ in $Q_{1,1}$, i.e.\ $u\le 1$ a.e.\ in $Q_{1,1}$.
  This proves \eqref{eq:B0}.

  \textbf{Step 7. Scaling back to $f$.}
  Undoing the intrinsic scaling $u=f/K$, with $\Theta=K^{2-p}$, as well as the preliminary translation and rescaling to $Q_{\Theta,R}(z_0)$ gives
  \[
    \esssup_{Q_{\Theta,R}(z_0)} f \le K
  \]
  provided \eqref{eq:S} holds. This completes the proof.
\end{proof}

Let us now consider the case $p <2$. Here, we iterate the truncated $\L^2$ norm.

\begin{theorem}\label{thm:DeGiorgi-localbdd-singular}
  Let $p \in (2-\frac{2}{3d+2},2)$ and assume \eqref{eq:a-bounds}.  Let $f\ge 0$ be
  a weak subsolution of \eqref{eq:nonlinear-laplace} in an open set containing
  $\overline{Q_{\Theta,2R}(z_0)}$. Then there exists a constant
  $\varepsilon_0=\varepsilon_0(d,p,\lambda,\Lambda)\in(0,1)$ such that the following holds.

  \medskip
  \noindent\textbf{Intrinsic boundedness principle.}
  Fix any level $K>0$ and set the intrinsic factor
  \begin{equation*}
    \Theta := K^{2-p}.
  \end{equation*}
  If $f$ satisfies the normalised smallness condition
  \begin{equation}\label{eq:S-singular-p2}
    \fint_{Q_{\Theta,2R}(z_0)}\Bigl(\frac{f}{K}\Bigr)^2 \dx (t,x,v)
    \le\varepsilon_0,
  \end{equation}
  then
  \begin{equation*}
    \esssup_{Q_{\Theta,R}(z_0)} f \le K.
  \end{equation*}
  In particular, nonnegative weak subsolutions are locally bounded.
\end{theorem}

\begin{proof}
  The proof is very similar to the previous proof.  The key singular feature for
  $p<2$ is that the term $\|\nabla_v(\cdot)\|_{\L^p}^{\,p-1}$ in Lemma~\ref{lem:gainsubloc}
  dominates for small truncations.  Again the proof is a De~Giorgi iteration, now in
  $\L^2$, combining (i) the energy estimate (Lemma~\ref{lem:energy}) on nested
  kinetic $p$-cylinders, and (ii) the localised gain of integrability
  (Lemma~\ref{lem:gainsubloc}).

  \textbf{Step 0. Normalisation.}
  As in the proof of Theorem~\ref{thm:DeGiorgi-localbdd}, we may assume that $z_0=0$
  and $R=1$. Thus, as before, we work on $Q_{\Theta,2}$, and the goal is to prove
  \[
    \esssup_{Q_{\Theta,1}} f \le K
    \quad\text{under \eqref{eq:S-singular-p2}.}
  \]
  Applying the intrinsic scaling again
  \[
    u(t,x,v):=\frac{1}{K}\,f\bigl(\Theta t,\ \Theta x,\ v\bigr),
    \qquad \Theta = K^{2-p},
  \]
  we see that $u$ is a weak subsolution of the structurally equivalent equation \eqref{TildeAEquation} in
  $Q_{1,2}$, where $\tilde{A}$ is as above. Moreover, a change of variables yields
  \[
    \fint_{Q_{1,2}} u^2\dx (t,x,v)
    =
    \fint_{Q_{\Theta,2}} \Bigl(\frac{f}{K}\Bigr)^2\dx (t,x,v).
  \]
  Hence \eqref{eq:S-singular-p2} becomes
  \begin{equation}\label{eq:S0-singular-p2}
    \fint_{Q_{1,2}} u^2\dx (t,x,v)
    \le \varepsilon_0.
  \end{equation}
  It suffices to prove the normalised claim
  \begin{equation}\label{eq:B0-singular-p2}
    \esssup_{Q_{1,1}} u \le 1.
  \end{equation}

  \textbf{Step 1. De Giorgi sequences of levels and radii.}
  Define radii and levels
  \[
    R_n := 1+2^{-n},\qquad k_n := 1-2^{-n},\qquad n=0,1,2,\dots
  \]
  so $R_0=2$, $R_n\downarrow 1$, and $k_0=0$, $k_n\uparrow 1$.
  Let
  \[
    w_n := (u-k_n)_+.
  \]
  Standard truncation arguments show that $w_n$ is also a weak subsolution.
  We define the $\L^2$ energies by
  \[
    Y_n := \|w_n\|_{\L^2(Q_{1,R_n})}^2
    = \int_{Q_{1,R_n}} w_n^2\dx (t,x,v).
  \]
  We will show that $Y_n\to 0$ if $\varepsilon_0$ is small enough. Since $w_n\downarrow (u-1)_+$ pointwise, this implies $(u-1)_+=0$
  a.e.\ in $Q_{1,1}$, i.e.\ \eqref{eq:B0-singular-p2}.

  \textbf{Step 2. The De Giorgi measure step (using $\L^q$).}
  Let
  \[
    \delta_n := k_{n+1}-k_n = 2^{-(n+1)}.
  \]
  On $\{w_{n+1}>0\}\cap Q_{1,R_{n+1}}=\{u>k_{n+1}\}\cap Q_{1,R_{n+1}}$ we have $w_n=u-k_n\ge \delta_n$, hence
  \[
    w_{n+1}\le w_n\,\mathds{1}_{\{w_n\ge \delta_n\}}
    \qquad\text{on }Q_{1,R_{n+1}}.
  \]
  Therefore
  \[
    Y_{n+1}
    =\int_{Q_{1,R_{n+1}}} w_{n+1}^2\dx (t,x,v)
    \le \int_{Q_{1,R_{n+1}}} w_n^2\,\mathds{1}_{\{w_n\ge \delta_n\}}\dx (t,x,v).
  \]
  By H\"older's inequality with exponents $\frac{q}{2}$ and $\bigl(\frac{q}{2}\bigr)^{\prime}$,
  \begin{equation}
    Y_{n+1}
    \le
    \bigl|\{w_n\ge \delta_n\}\cap Q_{1,R_{n+1}}\bigr|^{1-\frac{2}{q}}
    \,\|w_n\|_{\L^q(Q_{1,R_{n+1}})}^2.
    \label{eq:DG-measure-sing}
  \end{equation}
  By Chebyshev's inequality,
  \begin{equation}
    \bigl|\{w_n\ge \delta_n\}\cap Q_{1,R_{n+1}}\bigr|
    \le \delta_n^{-2}\int_{Q_{1,R_{n+1}}} w_n^2\dx (t,x,v)
    \le \delta_n^{-2} Y_n.
    \label{eq:chebyshev-sing}
  \end{equation}
  Combining \eqref{eq:DG-measure-sing}--\eqref{eq:chebyshev-sing} yields
  \begin{equation}
    Y_{n+1}
    \le
    \delta_n^{-2\eta}\,Y_n^{\eta}\,\|w_n\|_{\L^q(Q_{1,R_{n+1}})}^2,
    \qquad \eta:=1-\frac{2}{q}.
    \label{eq:DG-link-sing}
  \end{equation}

  \textbf{Step 3. Gradient bound from the energy estimate.}
  For $n\ge 1$, apply Lemma~\ref{lem:energy} to $w_n$ in the nested cylinders
  $Q_{1,R_{n-1}}\supset Q_{1,R_n}$ (take $R_1=R_n$, $R_2=R_{n-1}$, $\theta=1$):
  \begin{equation*}
    \|\nabla_v w_n\|_{\L^p(Q_{1,R_n})}^p
    \le
    \frac{C}{(R_{n-1}-R_n)^p}\,\|w_n\|_{\L^2(Q_{1,R_{n-1}})}^2
    +\frac{C}{(R_{n-1}-R_n)^p}\,\|w_n\|_{\L^p(Q_{1,R_{n-1}})}^p.
  \end{equation*}
  Since $w_n\le w_{n-1}$, we have
  \[
    \|w_n\|_{\L^2(Q_{1,R_{n-1}})}^2 \le Y_{n-1}.
  \]

  On $\{w_n>0\}$ we have $w_{n-1} \ge 2^{-n}$, hence
  \[
    |\{w_n>0\} \cap Q_{1,R_{n-1}} | \le 2^{2n} Y_{n-1}.
  \]
  We deduce
  \begin{equation} \label{eq:newLpL2}
    \|w_n\|_{\L^p(Q_{1,R_{n-1}})}^p \le \|w_n\|_{\L^2(Q_{1,R_{n-1}})}^p \, |\{w_n>0\}|^{1-\frac{p}{2}} \lesssim 2^{n(2-p)}Y_{n-1}.
  \end{equation}
  Therefore,
  \begin{equation*}
    \|\nabla_v w_n\|_{\L^p(Q_{1,R_n})}^p
    \le
    C\,2^{np}\Bigl(Y_{n-1}+2^{n(2-p)}Y_{n-1}\Bigr) \lesssim 2^{2n}Y_{n-1},
  \end{equation*}
  and thus
  \begin{equation}
    \|\nabla_v w_n\|_{\L^p(Q_{1,R_n})}
    \le
    C\,2^{\frac{2n}{p}}\,Y_{n-1}^{1/p},
    \qquad n\ge 1.
    \label{eq:grad-clean-sing}
  \end{equation}

  \textbf{Step 4. $\L^q$ estimate from the gain of integrability.}
  We assume from now on $Y_0\le 1$.
  Since $w_n\le w_{n-1}$ and $Q_{1,R_n}\subset Q_{1,R_{n-1}}$, we have $Y_n\le Y_{n-1}\le 1$ for all $n$.

  Apply Lemma~\ref{lem:gainsubloc} to the subsolution $w_n$ with $R_2=R_n$, $R_1=R_{n+1}$, and $\theta=1$.
  Since $R_n\in[1,2]$, the cutoff costs satisfy
  \begin{equation*}
    \Gamma_t \le C\,2^n,\qquad \Gamma_v\le C\,2^n,
  \end{equation*}
  and $|Q_{1,R_n}|\le C$ uniformly in $n$.
  Hence,
  \begin{align*}
     & \|w_n\|_{\L^q(Q_{1,R_{n+1}})} \\
     & \le C\Bigl[
      \|\nabla_v w_n\|_{\L^p(Q_{1,R_n})}
      +2^n\|\nabla_v w_n\|_{\L^p(Q_{1,R_n})}^{p-1}
      +2^n\|w_n\|_{\L^p(Q_{1,R_n})}
      \Bigr].
  \end{align*}
  Using \eqref{eq:newLpL2} and \eqref{eq:grad-clean-sing}, we obtain
  \[
    \|w_n\|_{\L^q(Q_{1,R_{n+1}})}
    \le C\Bigl[
      2^{\frac{2n}{p}}\,Y_{n-1}^{1/p}
      +2^{n}\bigl(2^{\frac{2n}{p}}\,Y_{n-1}^{1/p}\bigr)^{p-1}
      +2^{\frac{2n}{p}}Y_{n-1}^{1/p}
      \Bigr]
    \le C\,2^{3n}\,Y_{n-1}^{(p-1)/p},
  \]
  where in the last step we used $Y_{n-1}\le 1$ and $\frac{p-1}{p} \le \frac{1}{p}$ for $p \le 2$.

  Thus,
  \begin{equation}
    \|w_n\|_{\L^q(Q_{1,R_{n+1}})}
    \le C\,2^{3n}\,Y_{n-1}^{(p-1)/p},
    \qquad n\ge 1.
    \label{eq:Lq-sing-final}
  \end{equation}

  \textbf{Step 5. De Giorgi recursion for $Y_n$.}
  Inserting \eqref{eq:Lq-sing-final} into \eqref{eq:DG-link-sing} gives
  \begin{equation*}
    Y_{n+1}
    \le
    C\,\delta_n^{-2\eta}\,Y_n^{\eta}\,\bigl(2^{3n}\bigr)^2\,Y_{n-1}^{2(p-1)/p}.
  \end{equation*}
  Since $Y_n\le Y_{n-1}$ and $\eta>0$, we have $Y_n^{\eta}\le Y_{n-1}^{\eta}$, hence
  \begin{equation}
    Y_{n+1}
    \le
    C\,\delta_n^{-2\eta}\,2^{6n}\,Y_{n-1}^{2(p-1)/p+\eta}
    =
    C\,\delta_n^{-2\eta}\,2^{6n}\,Y_{n-1}^{s},
    \qquad s:=\frac{2(p-1)}{p}+1-\frac{2}{q}.
    \label{eq:Yn-rec-sing}
  \end{equation}
  With $\delta_n=2^{-(n+1)}$, we have $\delta_n^{-2\eta}=2^{2\eta(n+1)}$, so \eqref{eq:Yn-rec-sing} implies
  \begin{equation}
    Y_{n+1}
    \le
    C\,2^{\gamma n}\,Y_{n-1}^{s},
    \qquad n\ge 1,
    \label{eq:Yn-rec-clean-sing}
  \end{equation}
  for some $\gamma=\gamma(d,p)>0$.

  \textbf{Step 6. Convergence of the iteration.}
  As
  \[
    s-1
    =
    \frac{2(p-1)}{p}-\frac{2}{q}
    =
    \frac{(3d+2)p-(6d+2)}{p(2d+1)}
  \]
  it follows that
  \[
    s>1 \iff p> 2-\frac{2}{3d+2}.
  \]
  Set $\delta:=s-1>0$.
  Define the subsequences $B_m:=Y_{2m}$ and $C_m:=Y_{2m+1}$. From \eqref{eq:Yn-rec-clean-sing},
  \[
    B_{m+1}\le C\,2^{2\gamma m}\,B_m^{1+\delta},
    \qquad
    C_{m+1}\le C\,2^{\gamma(2m+1)}\,C_m^{1+\delta}.
  \]
  Both are of the form $Z_{m+1}\le C_1 b^m Z_m^{1+\delta}$ with $b>1$.
  A standard fast-convergence lemma yields $\delta_0=\delta_0(C_1,b,\delta)>0$ such that if $Z_0\le\delta_0$,
  then $Z_m\to 0$.
  Thus, if $Y_0\le\delta_0$, then $Y_n\to 0$.

  Choose $\varepsilon_0$ in \eqref{eq:S0-singular-p2} so small that
  \[
    Y_0=\int_{Q_{1,2}} u^2\dx (t,x,v)
    \le |Q_{1,2}|\fint_{Q_{1,2}} u^2\dx (t,x,v)
    \le |Q_{1,2}|\,\varepsilon_0
    \le \min\{1,\delta_0\}.
  \]
  Then $Y_n\to 0$ and, since $w_n\downarrow (u-1)_+$ pointwise on $Q_{1,1}$,
  \[
    0=\lim_{n\to\infty}Y_n
    =\int_{Q_{1,1}} (u-1)_+^2\dx (t,x,v),
  \]
  so $(u-1)_+=0$ a.e.\ in $Q_{1,1}$, i.e.\ $u\le 1$ a.e.\ in $Q_{1,1}$.
  This proves \eqref{eq:B0-singular-p2}.

  \textbf{Step 7. Scaling back to $f$.}
  Undoing the intrinsic scaling $u=f/K$, with $\Theta=K^{2-p}$, as well as the preliminary translation and rescaling to $Q_{\Theta,R}(z_0)$ yields
  \[
    \esssup_{Q_{\Theta,R}(z_0)} f \le K
  \]
  provided \eqref{eq:S-singular-p2} holds. This completes the proof.
\end{proof}

\bibliographystyle{plain}
\bibliography{kinetic_moser}

\end{document}